\documentclass[11pt,a4paper]{amsart}

\usepackage[utf8]{inputenc}
\usepackage{amsmath}
\usepackage[foot]{amsaddr}
\usepackage{mathtools}
\usepackage{amsfonts}
\usepackage{bbm} 
\usepackage{amssymb}
\usepackage{amsthm}
\usepackage{geometry}
\usepackage[noadjust]{cite}
\usepackage[usenames,dvipsnames]{xcolor}
\usepackage[vcentermath]{genyoungtabtikz}
\usepackage{tikz}
\usepackage{subcaption}
\usepackage[colorlinks, linkcolor=violet, citecolor=violet, urlcolor=violet]{hyperref}
\usepackage[pagewise]{lineno}

\usepackage[final]{showlabels}

\theoremstyle{plain}
\newtheorem{proposition}{Proposition}[section]
\newtheorem{theorem}[proposition]{Theorem}
\newtheorem{corollary}[proposition]{Corollary}
\newtheorem{lemma}[proposition]{Lemma}

\theoremstyle{definition}
\newtheorem{definition}[proposition]{Definition}
\newtheorem{example}[proposition]{Example}
\newtheorem{remark}[proposition]{Remark}

\newcommand{\Sn}{\mathfrak{S}_n}
\newcommand{\C}{\mathbb{C}}
\newcommand{\Reg}[2]{\textup{Reg}_{#1}(#2)}
\newcommand{\bg}[2]{\textup{BG}_{#1}^{#2}}

\newcommand\ten{10}
\newcommand\eleven{11}
\newcommand\twelve{12}

\newcommand{\equis}{\text{$\times$}}

\newcommand\fg{\Yfillcolour{black!20}}
\newcommand\fgg{\Yfillcolour{black!45}}
\newcommand\fggg{\Yfillcolour{black!75}}

\newcommand\fw{\Yfillcolour{white}}

\newcommand{\lt}{\Ylinethick{1.5pt}}
\newcommand{\lno}{\Ylinethick{0.4pt}}

\newcommand{\bgs}[1]{\textup{bg}_{#1}}

\author{Ana Bernal}
\address{Laboratoire de Math\'ematiques de Reims UMR 9008, Université de Reims Champagne Ardenne, Moulin de la Housse BP 1039, 51687 REIMS cedex 2, France}
\email{ana.bernal@univ-reims.fr}

\title{On self-Mullineux and self-conjugate partitions}

\begin{document}

\begin{abstract}
The Mullineux involution is a relevant map that appears in the study of the modular representations of the symmetric group and the alternating group. The fixed points of this map  are certain partitions of particular interest. It is known that the cardinality of the set of these \emph{self-Mullineux} partitions is equal to the cardinality of a distinguished subset of self-conjugate partitions. In this work, we give an explicit bijection between the two families of partitions in terms of the Mullineux symbol. \\
\end{abstract}

\maketitle

\section{Introduction}
Let $n$ be a non negative integer. It is well known that the isomorphism classes of complex irreducible representations of the symmetric group $\Sn$ is indexed by the set of partitions of $n$. Let $\lambda$ be a partition of $n\geq 2$ (written $\lambda \vdash n$) and $S^\lambda$ the associated irreducible $\C \Sn$-module. Tensoring $S^\lambda$ with the sign representation $\varepsilon$ of $\Sn$ results in the irreducible representation $S^{\lambda'}$ of $\Sn$, where $\lambda'$ is the conjugate partiton of $\lambda$ (\cite[2.1.8]{jameskerber}). This procedure allows to understand, by Clifford theory, all complex irreducible representations of the alternating group $A_n$. Indeed, let $\lambda$ be a partition of $n\geq 2$,

\begin{itemize}
\item If $\lambda \neq \lambda'$ then $S^\lambda \downarrow_{A_n} \simeq S^{\lambda'} \downarrow_{A_n}$ is irreducible.
\item If $\lambda = \lambda'$ then $S^\lambda \downarrow_{A_n}$ splits into two irreducible, non-isomorphic $\C A_n$-modules $S^\lambda_+$ and $S^\lambda_-$
\[
S^\lambda \downarrow_{A_n} \simeq S^\lambda_+ \oplus S^\lambda_-,
\]
\end{itemize}
and 
\[ 
\{ S^\lambda \downarrow_{A_n} \mid \lambda \vdash n\ \text{and}\  \lambda \neq \lambda' \}\  \sqcup \ \{S^\lambda_+,S^\lambda_- \mid \lambda \vdash n\ \text{and}\   \lambda = \lambda'\}
\]
is a complete set of non-isomorphic irreducible $\C A_n$-modules, considering only one partition $\lambda$ for each couple $\{\lambda,\lambda'\}$ with $\lambda \neq \lambda'$ (\cite[2.5.7]{jameskerber}).

A natural question is then what happens when we change the characteristic of the field. Let $p$ be an odd prime and $F$ an algebraically closed field of characteristic $p$. It is well known that the number of isomorphism classes of irreducible representations of the symmetric group $\Sn$ over $F$, is equal to the number of conjugacy classes of \emph{$p$-regular elements} of $\Sn$ (\cite[15.11]{isaacs}), which in turn is in bijection with the \emph{$p$-regular partitions} of $n$ (\cite[6.1.2]{jameskerber}). Also in this setting, understanding the tensor product with the sign representation allows to obtain a classification of irreducible $F A_n$-modules. However, the conjugate of a $p$-regular partition is not necessarily $p$-regular, so tensoring with the sign representation in this case does not amount to conjugating the corresponding $p$-regular partition.

In \cite{mulli}, G. Mullineux defined a bijection $m$ on the set of $p$-regular partitions of $n$, which is an involution, and conjectured that for a $p$-regular partition $\lambda$ with associated irreducible $F \Sn$-module $D^\lambda$ we have
\[
D^\lambda \otimes \varepsilon = D^{m(\lambda)}.
\]

Later, in \cite{kleshchev}, A. Kleshchev described a different algorithm to compute $m(\lambda)$ and in \cite{fordkleschev}, B. Ford and A. Kleshchev proved Mullineux conjecture to be true. Mullineux conjecture was also proven to be true in \cite{bessenrodtolsson2} by C. Bessenrodt and J.B. Olsson by using yet another description of the Mullineux bijection $m$. Other properties of this map have been studied for example in \cite{mullineux2}, \cite{bessenrodtolsson}, \cite{bessenrodtolssonxu}.
Hence, tensoring with the sign representation in the modular case amounts to applying $m$ on partitions, which makes the Mullineux map a $p$-analogue of conjugation of partitions. This way we have a classification of irreducible representations of $A_n$ in characteristic $p$ as follows. Let $\lambda$ be a $p$-regular partition of $n\geq 2$,

\begin{itemize}
\item If $\lambda \neq m(\lambda)$ then $D^\lambda \downarrow_{A_n} \simeq D^{m(\lambda)} \downarrow_{A_n}$ is irreducible.
\item If $\lambda = m(\lambda)$ then $D^\lambda \downarrow_{A_n} $ splits into two irreducible, non-isomorphic $F A_n$-modules $D^\lambda_+$ and $D^\lambda_-$
\[
D^\lambda \downarrow_{A_n} \ \simeq\  D^\lambda_+ \oplus D^\lambda_-,
\]
\end{itemize}
and 
\begin{equation}\label{eqn:indxmodA}
\{ D^\lambda \downarrow_{A_n} \mid \lambda \vdash n,\ \text{$\lambda$ $p$-regular and}\   \lambda \neq m(\lambda) \}\  \sqcup \ \{D^\lambda_+,D^\lambda_- \mid \lambda \vdash n,\ \text{$\lambda$ $p$-regular and}\   \lambda = m(\lambda)\}
\end{equation} 

is a complete set of non-isomorphic irreducible $F A_n$-modules, considering only one partition $\lambda$ for each couple $\{\lambda,m(\lambda)\}$ with $\lambda \neq m(\lambda)$ (\cite[2.1]{ford}). Following such an indexing of irreducible modular representations of $A_n$, it is natural to inquire about the set of $p$-regular partitions such that $\lambda=m(\lambda)$. The definition of the Mullineux map $m$ is quite complicated combinatorially, as are the different descriptions mentioned above, even if they are explicit. Therefore describing its fixed points is not easy.  So that, in characteristic $p$, it is not straightforward to obtain a reasonably simple indexing set for the irreducible $F A_n$-modules.

In fact, the number of fixed points of the Mullineux map, or \emph{self-Mullineux} partitions, is equal to the number of partitions of $n$ with different odd parts, none of them divisible by $p$ (\cite[Proposition 2]{andrewsolsson}). This number is, in turn, equal to the number of self-conjugate partitions with diagonal hook-length not divisible by $p$ (\cite[2.5.11]{jameskerber}). We refer to the latter as \emph{BG-partitions} (see Definition \ref{def:bgpart} for details). There is an elementary algebraic argument to see this (Appendix \ref{appendix}). Thus, it is natural to ask for an explicit bijection between the self-Mullineux partitions and the BG-partitions. 

The Mullineux map can be defined in terms of a symbol called the Mullineux symbol, defined on $p$-regular partitions. In this work we introduce a new symbol, defined on self-conjugate partitions. From such a symbol, associated to a BG-partition, we describe how to reconstruct a BG-partition and a self-Mullineux partition, and this algorithm provides our bijection.

 A further motivation for finding an explicit bijection can also be given in the context of the representation theory of the symmetric group and of the alternating group. In \cite{brunatgramain}, O. Brunat and J.-B. Gramain have shown the existence of a $p$-basic set for the symmetric group, which, by restriction, gives a $p$-basic set for the alternating group. However, this set, which provides a natural indexing set for the modular irreducible representations is not explicit and it would be ideal to give a complete description of it. One thing we know about such a set is that it always contains the set of BG-partitions. Hence, it is convenient to have a better understanding of them. More generally, this work can be seen as a first step to give a new natural way to label the modular irreducible representations of the symmetric group, for which tensoring with the sign representation is easier to describe combinatorially. We hope to come back to this problem later. 

A bijection between the set of self-Mullineux partitions of $n$ and partitions of $n$ with different odd parts, none of them divisible by $p$ can alternatively be derived from a bijection between two more general sets defined by C. Bessenrodt in \cite{bessenrodt}. However, the two approaches are quite different because our bijection is defined directly between the sets of our interest. Moreover, we obtain a different bijection (see Remark \ref{rem:bess}).

The paper is organized as follows. In Section \ref{secmulli} we recall some definitions about partitions and the definition of the Mullineux map. Section \ref{subsec:bg} contains the main result of this paper; we define a symbol on self-conjugate partitions and we show how through this symbol we obtain the mentioned explicit bijection.
Finally, in Section \ref{sect:surj} we prove that the inverse procedure of reconstructing a unique BG-partition from a self-Mullineux partition is well defined, which confirms that this is a one to one correspondence without the need of knowing beforehand that there exists a bijection between both sets of partition.

\medskip
\noindent \textbf{Acknowledgements.}  The author wishes to thank her advisors Nicolas Jacon and Lo\"ic Poulain d'Andecy for their precious advice, helpful discussions and careful reading. The author would also like to thank an anonymous referee for useful suggestions.

\section{Preliminaries: the Mullineux map} \label{secmulli}
In this section we recall some general definitions about partitions and the definition of the Mullineux map, as defined by G. Mullineux in \cite{mulli}. We follow closely definitions in \cite{fordkleschev}.

A \emph{partition} $\lambda$ is a sequence $\lambda = (\lambda_1,\lambda_2,\ldots,\lambda_k,\ldots)$ of non-negative integers such that $\lambda_1 \geq\lambda_2 \geq\ \cdots \geq \lambda_k \geq \cdots $, containing only finitely many non-zero terms. Let $n \in \mathbb{N}$ be such that $|\lambda| =\sum \lambda_i = n$. We say that $\lambda$ is a partition of $n$, which we write $\lambda \vdash n$. We note $k(\lambda)=\textup{max}\{i \mid \lambda_i \geq i\}$. Let $\textup{Par}(n)$ denote the set of partitions of $n$. We call $|\lambda|$ the \emph{size} of $\lambda$. The integers $\lambda_i$ are called the \emph{parts} of the partition $\lambda$. The number of non-zero parts is the \emph{length} of $\lambda$ and is denoted $l(\lambda)$.
The \emph{Young diagram} of a partition $\lambda$ is the set \[
[\lambda]=\{(i,j) \in \mathbb{N} \times \mathbb{N} \mid i\geq 1  \ \ \text{and} \ \ 1 \leq j \leq \lambda_i \},
\] 
whose elements are called \emph{nodes}. We represent the diagram as an array of boxes in the plane with the convention that $i$ increases downwards and $j$ increases from left to right. A partition is often identified with its Young diagram. The Young diagram of $\lambda=(5,2^2,1)$ is 
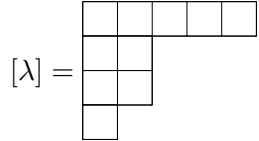
\[
[\lambda]=\yngs(1,5,2,2,1)
\]
The \emph{conjugate} (or \emph{transpose}) partition of $\lambda=(\lambda_1,\ldots,\lambda_k)$ is the partition $\lambda'$ of $n$ defined as $\lambda'_j=\#\{i \mid \lambda_i \leq j\}$, which amounts to transposing the Young diagram $[\lambda]$ with respect to its main diagonal which consists of the nodes of the form $(i,i)$ with $1\leq i\leq k(\lambda)$. If $\lambda=(5,2^2,1)$, as above, then $\lambda'=(4,3,1^3)$ and its Young diagram is 
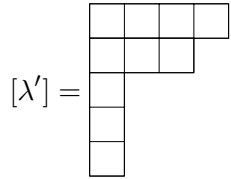
\[
[\lambda']=\yngs(1,4,3,1^3)
\]
If $\lambda=\lambda'$ we say that $\lambda$ is \emph{self-conjugate}. For a positive integer $p$, $\lambda$ is said to be \emph{$p$-regular} if it does not contain $p$ parts $\lambda_i \neq 0$ which are equal. The partition $\lambda=(5,2^2,1)$ above is not $2$-regular but it is $3$-regular. We denote by $\Reg{p}{n}$ the set of $p$-regular partitions of $n$. 

Let $(i,j)$ be a node of a partition $\lambda$. We define the $(i,j)$\emph{-th hook} of $\lambda$, or the $(i,j)_\lambda$-th hook, as the set of nodes in $[\lambda]$ to the right or below the node $(i,j)$, that is, the nodes $(k,l)$ such that $k=i$ and $j \leq l \leq \lambda_i$, or $l=j$ and $i \leq k \leq \lambda'_j$. The \emph{hook-length} of $\lambda$ at $(i,j) \in [\lambda]$, denoted here $h^\lambda_{ij}$ is the number of nodes in the $(i,j)_\lambda$-th hook, that is
\[
h^\lambda_{ij}=\lambda_i+\lambda'_j-i-j+1.
\]
We omit $\lambda$ from the notation when there is no ambiguity. A partition which is equal to its $(1,1)$-th hook is called a \emph{hook.}

\begin{example} Let $\lambda=(7,3,2,1)$. The $(1,2)$-th hook of $\lambda$ is represented by shaded boxes in the following diagram
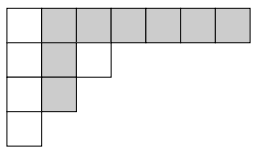
\[
\gyoung(;!\fg;;;;;;,!\fw;!\fg;!\fw;,;!\fg;,!\fw;)
\]
and its length is $h_{1,2}=8$. 
\end{example}
\medskip
For any positive integer $m$, a \emph{m-hook} (respectively \emph{$(m)$-hook}) is a hook of length $m$ (respectively divisible by $m$). We call a node $(i,i)$, for $1\leq i \leq k(\lambda)$, a \emph{diagonal} node and the set of diagonal nodes is the \emph{diagonal} of $\lambda$. A $(i,i)$-th hook is referred to as a \emph{diagonal hook}. 

The \emph{rim} of $\lambda$ is the set of nodes $(i,j) \in [\lambda]$ such that $(i+1,j+1) \notin [\lambda]$, in other words, it is the south-east border of the Young diagram $[\lambda]$. For example, the rim of $\lambda=(9,6,3,1)$ is represented in the following diagram by shaded boxes

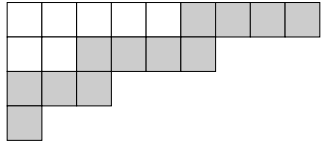
\[
\gyoung(;;;;;!\fg;;;;,!\fw;;!\fg;;;;,;;;,;)
\]

Let us label the nodes of the rim with positive integers from the top right to the bottom left, as shown in the following figure

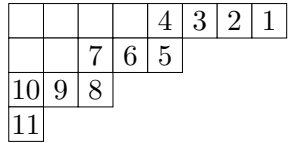
\[
\gyoung(;;;;;4321,;;;765,\ten98,\eleven)
\]

Let $p$ be an odd prime. The first \emph{p-segment} of the rim consists of  the nodes corresponding to integers less or equal than $p$. If the last node $(i,j)$ of the first $p$-segment is in the last row of $[\lambda]$, then $[\lambda]$ only has one $p$-segment. If not, let $l$ be the smallest label on row $i+1$. The second $p$-segment of the rim consists of the nodes labelled by $l \leq m \leq l+p-1$. Repeating this procedure we will eventually reach the bottom row of the diagram and it is clear that all $p$-segments have $p$ nodes, except possibly the last one. The \emph{p-rim} of $\lambda$ is defined as the union of all the $p$-segments.

\begin{example} The following two diagrams illustrate the $p$-rim of $\lambda=(9,6,3,1)$ for $p=3$ and $p=5$.

\[
p=3 \ \ \gyoung(;;;;;;!\fg;;;,!\fw;;;!\fg;;;,;;;,;) \ \ \ \ \ \ p=5 \ \ \gyoung(;;;;;!\fg;;;;,!\fw;;;;;!\fg;,;;;,;)
\]
\end{example}
\medskip

 We denote $a_\lambda$ the number of nodes in the $p$-rim of $\lambda$. Define diagrams $\lambda^{(0)},\lambda^{(1)},\ldots,\lambda^{(l)}$ as follows. Put $\lambda^{(0)}=\lambda$ and for $i\geq1$ put
\[
\lambda^{(i)}= \lambda^{(i-1)} \setminus \{p\text{-rim of }\ \lambda^{(i-1)}\},
\]
where we choose $l$ maximal with respect to $\lambda^{(l)} \neq \emptyset$; so $\lambda^{(l+1)} = \emptyset$. We call the $p$-rim of $\lambda^{(i)}$ the $i$-th $p$-rim of $\lambda$. Let $a_i=a_{\lambda^{(i)}}$ be the number of nodes of the $i$-th $p$-rim of $\lambda$ and $r_i$ the number of rows of $\lambda^{(i)}$, that is, $r_i=l(\lambda^{(i)})$. The \emph{Mullineux symbol} of $\lambda$ (introduced in \cite{bessenrodtolsson}) is
\begin{equation}\label{mullisymb}
G_p(\lambda)=
\begin{pmatrix}
a_0 & a_1 & \cdots & a_l \\
r_0 & r_1 & \cdots & r_l
\end{pmatrix}.
\end{equation}

\begin{example}\label{ex-labels}
Let $p=5$ and $\lambda=(9,6,3,1)$. In the following diagram we represent the $i$-th $p$-rim of $\lambda$ with label $i$ on its nodes
\Yvcentermath1
\[
\young(222210000,211110,000,0) \ \ \ \ \  G_5(\lambda)= \begin{pmatrix}
9 & 5 & 5 \\
4 & 2 & 2
\end{pmatrix}.
\]

\end{example}
\medskip
\noindent The following proposition is a reformulation (\cite[\textsection 5]{andrewsolsson}) of a result proved in \cite[3.6]{mulli}.

\begin{proposition}\label{cinco} Let $p$ be an odd prime and $\lambda$ a $p$-regular partition of a non-negative integer $n$. Set  
\[ \varepsilon_i=
\begin{cases}
0\ & \text{if} \ p \mid a_i,\\
1\ & \text{if} \ p \nmid a_i.
\end{cases}
\]
The entries of $G_p(\lambda)$ satisfy
\begin{enumerate}
\item $\varepsilon_i \leq r_i-r_{i+1} < p+\varepsilon_i $ for $0 \leq i < l$,
\item $1 \leq r_l < p+\varepsilon_l$,
\item \label{cinco3} $r_i-r_{i+1}+\varepsilon_{i+1} \leq a_i-a_{i+1} < p+ r_i-r_{i+1}+\varepsilon_{i+1}$ for $0 \leq i < l$,
\item $r_l \leq a_l < p+r_l$,
\item $\sum_{i=0}^l a_i =n$.
\end{enumerate}
Moreover, if $a_0,\ldots,a_l,r_0,\ldots,r_l$ are positive integers such that these inequalities are satisfied then there exists exactly one $p$-regular partition $\lambda$ of $n$ such that
\[
G_p(\lambda)=
\begin{pmatrix}
a_0 & a_1 & \cdots & a_l \\
r_0 & r_1 & \cdots & r_l
\end{pmatrix}.
\]
\end{proposition}

\begin{remark}\label{RemConst} It is easy to recover the $p$-regular partition $\lambda$ from its Mullineux symbol $G_p(\lambda)$; start with the hook $\lambda^{(l)}$ of size $a_l$ and length $r_l$, and for $i=l-1,l-2,\ldots ,0$, add the $i$-th $p$-rim (consisting of $a_i$ nodes) to $\lambda^{(i+1)}$ from the bottom to the top, starting by placing a node on the first free placement in row $r_i$. Then, adding nodes either on top (whenever it is possible) or to the right of the last added node until having added the last node of the $p$-segment and add the following $p$-segment starting on the first free placement of the row on top of the last added node. This procedure finishes at the first row. This algorithm is more precisely described in \cite[\textsection 1]{fordkleschev}.
\end{remark}
\medskip
Let $\lambda$ be a $p$-regular partition of $n$, with Mullineux symbol (\ref{mullisymb}) and let $\varepsilon_i$ be as in Proposition \ref{cinco}. Define $s_i=a_i+\varepsilon_i-r_i$. In \cite[4.1]{mulli} it is shown that the array
\[
\begin{pmatrix}
a_0 & a_1 & \cdots & a_l \\
s_0 & s_1 & \cdots & s_l
\end{pmatrix},
\]
corresponds to the Mullineux symbol of a $p$-regular partition. We are now able to define the Mullineux map $m$.

\begin{definition} With the above notations, $m(\lambda)$ is defined as the unique $p$-regular partition such that
\[ G_p(m(\lambda))=
\begin{pmatrix}
a_0 & a_1 & \cdots & a_l \\
s_0 & s_1 & \cdots & s_l
\end{pmatrix}.
\]
\end{definition}

Because of Proposition \ref{cinco}, $m(\lambda)$ is well defined, and from the definition we can see that $m$ is an involution.

\begin{remark} If $p>n$, then $\Reg{p}{n}=\textup{Par}(n)$ and irreducible $F \Sn$-modules are therefore indexed by all partitions of $n$. In this case, the Mullineux map coincides with conjugation: $m(\lambda)=\lambda'$.
\end{remark}

\section{From BG-partitions to self-Mullineux partitions} \label{subsec:bg}
Let $p$ be an odd prime. Let us define the set of BG-partitions.
\begin{definition} \label{def:bgpart}
We call \emph{BG-partition} any self-conjugate partition with no diagonal $(p)$-hooks, that is, any partition $\lambda$ such that $p \nmid h^\lambda_{ii}$ for every $1 \leq i \leq k(\lambda)$ with $(i,i) \in [\lambda]$. We denote $\bg{p}{}$ the set of BG-partitions and $\bg{p}{n}$ the set of BG-partitions of a non-negative integer $n$.
\end{definition}

As said in the introduction, the aim of this work is to give an explicit bijection between BG-partitions and self-Mullineux partitions. In this section we describe such bijection: we define the \emph{BG-symbol} associated to a BG-partition, which is a two-row array of positive integers. We prove that BG-symbols of BG-partitions are Mullineux symbols of self-Mullineux partitions and that this association is injective, resulting in a bijection.

\subsection{BG-symbol} 

We introduce a symbol, defined in general for self-conjugate partitions. This symbol is in some way inspired by the Mullineux symbol. In a similar way as the Mullineux symbol, which is defined by counting nodes on the $p$-rims of a sequence of partitions, the BG-symbol is defined by counting elements in a set of nodes called the $p$-rim* which is a \emph{symmetric} analogue of the $p$-rim.

Let $\lambda$ be a self-conjugate partition and denote by $\text{Rim}_p(\lambda) \subseteq [\lambda]$ the set of nodes in the $p$-rim of $[\lambda]$. Set
\[
U_\lambda=\{(i,j) \in \text{Rim}_p(\lambda) \mid i \leq j\},
\]
that is, $U_\lambda$ consists of the nodes of the $p$-rim which are above (or on) the diagonal of $[\lambda]$. We denote $r_\lambda^*:=\#\, U_\lambda$. Set
\[
L_\lambda=\{(j,i) \mid (i,j)\in U_\lambda\}.
\]
The set $L_\lambda$ consists of the nodes in $U_\lambda$ reflected across the diagonal of $\lambda$. Notice that $L_\lambda \subseteq [\lambda]$, since $\lambda=\lambda'$, so that $(i,j)\in[\lambda]$ if and only if $(j,i)\in[\lambda]$.

\begin{definition}\label{defrim*}
Let $\lambda$ be a self-conjugate partition. The \emph{p-rim*} of $\lambda$ is the set $
\textup{Rim}^*_p(\lambda)=U_\lambda \cup L_\lambda.$
We denote $a^*_\lambda$ the number of nodes in the $p$-rim* of $\lambda$, that is, $a^*_\lambda = \# \, \textup{Rim}^*_p(\lambda)$.
Define $\varepsilon^*_\lambda$ as $\varepsilon^*_\lambda = a^*_\lambda\ \textup{mod}\ 2
$. It is the parity of the number of nodes on the $p$-rim* of $\lambda$.
\end{definition}

\begin{example} The following two diagrams illustrate the $p$-rim* of $\lambda=(6,2,1^4)$ in shaded boxes, for $p=3$ and $p=5$.

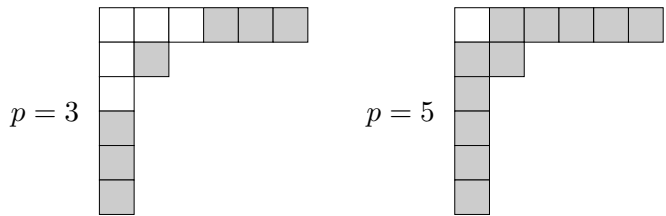
\[
p=3 \ \ \gyoung(;;;!\fg;;;,!\fw;!\fg;,!\fw;!\fg,;,;,;) \ \ \ \ \ \ p=5 \ \ \gyoung(;!\fg;;;;;,;;,;,;,;,;)
\]
\end{example}
\bigskip

\begin{remark} \label{parite2}
For a self-conjugate partition $\lambda$, from the definition of $p$-rim* we have
\[
\begin{array}{ccl}
\varepsilon^*_\lambda = 0  & \Leftrightarrow & a^*_\lambda \ \text{is even.} \\
& \Leftrightarrow & \textup{Rim}^*_p(\lambda)\  \text{has no diagonal nodes.}
\end{array}
\]
\end{remark}

This way, the number of nodes in $U_\lambda$, that is, the number of nodes of the $p$-rim* of $\lambda$ over (or on) the diagonal is
\[
r_\lambda^* = 
\begin{cases} 
\frac{a^*_\lambda}{2} & \text{if $a^*_\lambda$ is even,}\\
\frac{a^*_\lambda+1}{2} &  \text{otherwise,}
\end{cases}
\]
thus
\[
r_\lambda^* = \frac{a^*_\lambda + \varepsilon^*_\lambda}{2}.
\]

Let $\lambda$ be a self-conjugate partition. We define diagrams $\lambda^{(0)*},\lambda^{(1)*},\ldots,\lambda^{(l)*}$ in an analogue way as for the Mullineux symbol, by considering the $p$-rim* instead of the $p$-rim. Put $\lambda^{(0)*}=\lambda$ and for $i\geq1$ put
\[
\lambda^{(i)*}= \lambda^{(i-1)*} \setminus \{p\text{-rim* of }\ \lambda^{(i-1)*}\},
\]
where we chose $l$ maximal with respect to $\lambda^{(l)*} \neq \emptyset$; so $\lambda^{(l+1)*} = \emptyset$. We call the $p$-rim* of $\lambda^{(i)*}$ the \emph{$i$-th $p$-rim*} of $\lambda$.

\begin{remark}\label{removeprim}
Notice that the $p$-rim* is only defined for self-conjugate partitions, but we claim that the diagrams $\lambda^{(i)*}$ are well defined, given the fact that $\textup{Rim}^*_p(\lambda)$ is \emph{symmetric} in the sense that $(u,v)\in \textup{Rim}^*_p(\lambda)$ if and only if $(v,u)\in \textup{Rim}^*_p(\lambda)$. Therefore, removing these nodes from $[\lambda]$ to obtain $\lambda^{(1)*}$ results again in a self-conjugate partition and then so it is for every $\lambda^{(i)*}$. In other words, if $\lambda^{(i)*}$ is self-conjugate, then $\lambda^{(i+1)*}$ is self-conjugate.
\end{remark}

\begin{example}\label{exprims}
Let $p=3$ and $\lambda=(6,5^2,3^2,1)$. Then
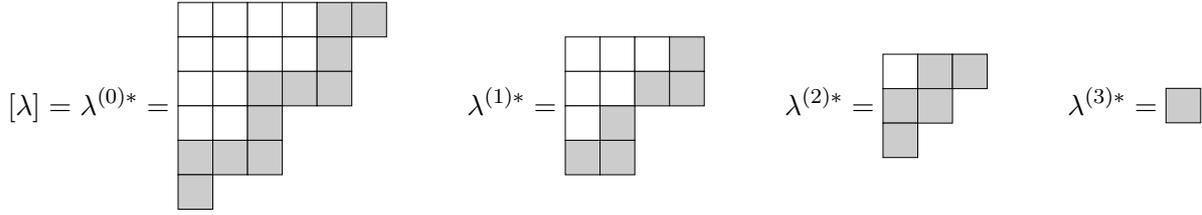
\[
[\lambda]=\lambda^{(0)*}=\gyoung(;;;;!\fg;;!\fw,;;;;!\fg;!\fw,;;!\fg;;;!\fw,;;!\fg;,;;;,;) 
\hspace{30pt}
\lambda^{(1)*}=\gyoung(;;;!\fg;!\fw,;;!\fg;;!\fw,;!\fg;,;;)
\hspace{30pt}
\lambda^{(2)*}=\gyoung(;!\fg;;,;;,;)
\hspace{30pt}
\lambda^{(3)*}=\gyoung(!\fg;)
\]

where shaded boxes represent the $i$-th $p$-rim* of $\lambda$. 

\end{example}
\medskip

 For the partition $\lambda^{(i)*}$, obtained by succesively removing nodes on the $p$-rim*, starting with $\lambda=\lambda^{(0)*}$, we denote $a^*_{\lambda^{(i)*}}$ as $a^*_i$. Similarly $r^*_{\lambda^{(i)*}}=r^*_i$ and $\varepsilon^*_{\lambda^{(i)*}}=\varepsilon^*_i$.
 All these values associated to self-conjugate partitions may seem technical, and they are better understood by means of the Young diagram, see the following example.
 
\begin{example}
Let $p=3$, $\lambda=(4^2,2^2)$, and $\mu=(3,2,1)$
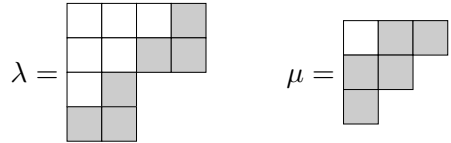
\[
\lambda=\gyoung(;;;!\fg;!\fw,;;!\fg;;!\fw,;!\fg;,;;)
\hspace{30pt}
\mu=\gyoung(;!\fg;;,;;,;)
\]
We have $a^*_\lambda=6$, $\varepsilon^*_\lambda=0$, and $r^*_\lambda=3$. For $\mu$, we have $a^*_\mu=5$, $\varepsilon^*_\mu=1$, and $r^*_\mu=3$.
\end{example}

\begin{definition} Let $\lambda$ be a self-conjugate partition.  The \emph{BG-symbol} of  $\lambda$ is
\begin{equation}
\bgs{p}(\lambda):=
\begin{pmatrix}
a^*_0 & a^*_1 & \cdots & a^*_l \\
r^*_0 & r^*_1 & \cdots & r^*_l
\end{pmatrix}.
\end{equation}
\noindent The \emph{length} of the BG-symbol is $l$.

\end{definition}

\begin{example} If $p=3$, the BG-symbol of the partition $\lambda=(6,5^2,3^2,1)$ is

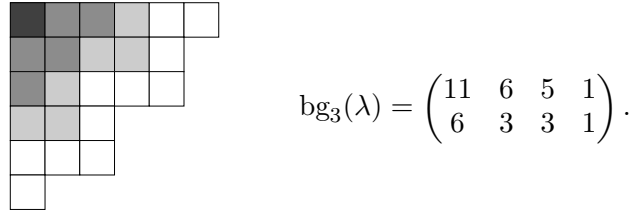
\[
\gyoung(!\fggg;!\fgg;;!\fg;!\fw;;!\fgg,;;!\fg;;!\fw;,!\fgg;!\fg;!\fw;;;!\fg,;;!\fw;,;;;,;) 
\hspace{30pt}
\bgs{3}(\lambda)=
\begin{pmatrix}
11 & 6 & 5 & 1 \\
6 & 3 & 3 & 1
\end{pmatrix}.
\]

\noindent In this diagram, each $i$-th $3$-rim* is shown in a different shade.

\end{example}
\medskip
The following lemmas will allow us to prove that two different self-conjugate partitions correspond to different BG-symbols. Lemma \ref{lema1} is an analogue of \cite[2.1]{mulli}. Its proof is quite technical and the arguments are easier to understand with an example, see Example \ref{examplelemma1}.

\begin{lemma}\label{parite}
Let $\lambda$ be a self-conjugate partition. If $a^*_\lambda$ is an even number, then $p \mid a^*_\lambda$.
\end{lemma}
\begin{proof}
From the definition (or see Remark \ref{parite2}), $a^*_\lambda$ is even if and only if $U_\lambda \cap L_\lambda = \emptyset$. Then the $p$-rim* of $\lambda$ does not contain diagonal nodes. From the definition of $\textup{Rim}^*_p(\lambda)$, this means that the set $U_\lambda$ only contains $p$-segments of length $p$. And then the same is true for $L_\lambda$. Therefore
\[
p \mid \#\,(U_\lambda \cup L_\lambda) =a^*_\lambda.
\] 
\end{proof}
The converse is not true in general, for example, if $p=3$ and $\lambda=(5,3,2,1,1)$, we have that $a^*_\lambda=9$.

\begin{lemma}\label{lema1}
Let $\tilde{\lambda}$ be a self-conjugate partition, $\varepsilon\in \{0,1\}$  and $m$, a residue modulo $p$, such that $m=0$ if $\varepsilon=0$. Then, there exists a unique self-conjugate partition $\lambda$ such that 
\begin{enumerate}
\item[(i)] $a^*_\lambda \equiv \varepsilon\ (\textup{mod}\ 2)$;
\item[(ii)] $r_\lambda^*-\varepsilon_\lambda^* \equiv m\ (\textup{mod}\ p)$ and 
\item[(iii)] $\lambda^{(1)*}=\tilde{\lambda}$.

\noindent Moreover, if $\tilde{\lambda} \in \bg{p}{}$, and $p \nmid 2m+1$ when $\varepsilon=1$, then $\lambda \in \bg{p}{}.$
\end{enumerate}

\end{lemma}

\begin{proof}

Given $\varepsilon \in \{0,1 \}$ and $m$, a residue modulo $p$, let us see that there is a unique way to add nodes to $\tilde{\lambda}$ to obtain a self-conjugate partition $\lambda$ such that the added nodes are the $p$-rim* of $\lambda$.

Let us study how nodes $(i,j)$ over the diagonal ($i\leq j$) must be added. This will determine all nodes that must be added (if $(i,j)$ is added to $\tilde{\lambda}$, then $(j,i)$ is added as well).

First, the last row $i$ over the diagonal that will contain new nodes $(i,j)$ is uniquely fixed by $\tilde{\lambda}$ and $\varepsilon$. Indeed, let $d=k(\tilde{\lambda})$. If $\varepsilon=0$, then $i=d$ and $(d,\tilde{\lambda}_d+1)$ must be added to $\tilde{\lambda}$. If $\varepsilon=1$, then $i=d+1$ and $(d+1,d+1)$ must be added to $\tilde{\lambda}$.

Now, let $(i,j)$ be the first node that we add (with $i$ fixed as before by $\tilde{\lambda}$ and $\varepsilon$) and $j\in \{\tilde{\lambda}_d+1,d+1\}$ depending on $\tilde{\lambda}$ and $\varepsilon$. Starting from this node, it is clear that there is a unique way to add nodes such that (i), (ii), and (iii) hold: If the position $(i+1,j)$ just above $(i,j)$ is empty in $\tilde{\lambda}$, we add a node in that position, otherwise we add a node in $(i,j+1)$. We repeat this procedure for adding nodes until we have added $m$ nodes (including $(i,j)$ if $\varepsilon=0$, not including $(i,j)$ if $\varepsilon=1$). If the last added node is in row $1$ we stop here. If it is added in row $k>1$, we add a node in row $k-1$ in position $(k-1,\tilde{\lambda}_{k-1}+1)$ and we restart the procedure to keep adding nodes until having added $p$ nodes. We iterate this procedure, of adding groups of $p$ nodes, until reaching the first row. This way we added nodes over the diagonal. Finally for each node $(a,b)$ added, we add its reflection through the diagonal $(b,a)$. And we obtain a self-conjugate partition $\lambda$.

It remains to verify that $\lambda^{(1)*}=\tilde{\lambda}$. If $\varepsilon=1$, it is straightforward that $\lambda^{(1)*}=\tilde{\lambda}$. Since when removing the nodes of the $p$-rim* of $\lambda$ over the diagonal we eventually reach a diagonal node, and then just remove the reflection of the removed nodes. It is clear that in this case  we obtain $\tilde{\lambda}$. If $\varepsilon=0$, the condition $m=0$ says that a $p$-segment of $\lambda$ will eventually reach the row $d$ and this $p$-segment has exactly $p$ nodes, so that there is no ambiguity when removing $p$-segments and $\lambda^{(1)*}=\tilde{\lambda}$.

For the last part of the theorem, suppose that $\tilde{\lambda} \in \bg{p}{}$, and let us see that $\lambda$ obtained as above is also in $\bg{p}{}$. In other words, we are assuming that $\tilde{\lambda}$ does not contain any diagonal $(p)$-hooks and we want to show that the same is true for $\lambda$.

Suppose that $\lambda$ has a diagonal $(p)$-hook, say the $(i,i)_\lambda$-th hook, that is $h^\lambda_{ii}=pk$ for some integer $k>0$. For a partition $\mu$, we set the convention $h^\mu_{ij}=0$ if $(i,j)\notin [\mu].$

Since $\tilde{\lambda} \in \bg{p}{}$, then the $(i,i)_\lambda$-th hook is different from the  $(i,i)_{\tilde{\lambda}}$-th hook since if they were equal, $\tilde{\lambda}$ would have a $(p)$-hook, which is not possible. Therefore $(i,\tilde{\lambda}_i+1) \in [\lambda]$. Since this node is not in $[\tilde{\lambda}]$, by definition, it is on the $p$-rim* of $\lambda$, in particular, it belongs to a $p$-segment of $\textup{Rim}_p^*(\lambda)$ above the diagonal. Consider the two cases: this $p$-segment starts at row $i$, or this $p$-segment starts before row $i$, that is, this $p$-segment starts at a row $j$ for $1 \leq j <i$.

\begin{itemize}
\item If the $p$-segment containing node $(i,\tilde{\lambda}_i+1)$ starts at row $i$, let $(i,j)$ be the first node of this $p$-segment and $(a,b)$ its last node ($i \leq a$). Then $a \leq b$ because this segment is above the diagonal.

Let $N$ be the number of nodes on this $p$-segment. Then we have:
\[
\begin{array}{ccc}
N & = &a-i+j-b+1, \\
h_{ii}^{\lambda}& =&1+2(j-i),\\
h_{aa}^{\tilde{\lambda}}& =&1+2(b-a-1),
\end{array}
\]
where the last identity holds if $b>a$ (since this implies that $(a,b-1)\in\tilde{\lambda}$). Then we have
\[h_{ii}^{\lambda}=
\begin{cases}
2N-1 & \text{if } a=b, \\
2N + h_{aa}^{\tilde{\lambda}} & \text{if } a<b.
\end{cases}
\]
If $a=b$, this $p$-segment is the last segment in $U_\lambda$ and $\varepsilon_\lambda^*=1$. So that $N=m+1$ and we get $h^\lambda_{ii}=2m+1$. This contradicts $p \nmid 2m+1$.
If $a<b$, then the last node of this $p$-segment, $(a,b)$ is not a diagonal node so that $N=p$ and we get $h^\lambda_{ii}=2p+h^{\tilde{\lambda}}_{aa}$, which implies $p \mid h^{\tilde{\lambda}}_{aa}$, a contradiction.

\item  If the $p$-segment containing node $(i,\tilde{\lambda}_i+1)$ starts at a row $j$ with $j<i$, then it contains  nodes on row $i-1$, in particular $(i-1,\tilde{\lambda}_{i-1}+1)$. The next node on this $p$-segment is the node just below: $(i,\tilde{\lambda}_{i-1}+1)$. Then $\lambda_i=\tilde{\lambda}_{i-1}+1$. Let us see that $h^{\tilde{\lambda}}_{(i-1,i-1)}=h^\lambda_{ii}$. Since these are diagonal hooks contained in self-conjugate partitions, their lengths are calculated as follows

\begin{align*}
h^{\tilde{\lambda}}_{(i-1,i-1)} &=2(\tilde{\lambda}_{i-1}-(i-1))+1 \\
						&=2\tilde{\lambda}_{i-1} -2i+3,
\end{align*}
and 
\begin{align*}
h^{\lambda}_{(i,i)} &=2(\lambda_{i}-i)+1 \\
					&=2((\tilde{\lambda}_{i-1}+1)-i)+1 \\
					&=2\tilde{\lambda}_{i-1} -2i+3.
\end{align*}

Since $p \mid h^{\lambda}_{(i,i)}$, then $p \mid h^{\tilde{\lambda}}_{(i-1,i-1)}$, a contradiction.

%
%
\end{itemize}
 In conclusion, $\lambda$ does not have any diagonal $(p)$-hooks, that is, $\lambda \in \bg{p}{}$.

\end{proof}

\begin{example}\label{examplelemma1}
Let $p=3$. We use the notations of Lemma \ref{lema1}. Consider the self-conjugate partition $\tilde{\lambda}=(6,4,2^2,1^2)$. 

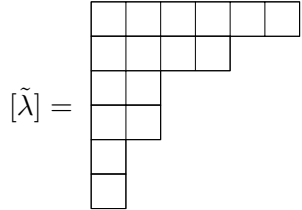
\Yvcentermath1
\[
[\tilde{\lambda}]=\ \yng(6,4,2,2,1,1)
\]
 \begin{itemize}
 \item[-] Let $\varepsilon=0$, then $m=0$. Let us see that there is only one self-conjugate partition $\lambda$ satisfying: $a_\lambda^*$ is even,
$r_\lambda^*-\varepsilon_\lambda^* =\#\,U_\lambda - \varepsilon_\lambda^* \equiv 0\ (\textup{mod}\ 3)$ and 
 $\lambda^{(1)*}=\tilde{\lambda}$. We add to $[\tilde{\lambda}]$ the nodes of $\text{Rim}_3^*(\lambda)=U_\lambda\cup L_\lambda$.

In this case, since $a^*_\lambda \equiv 0\ (\text{mod}\ 2)$, then $U_\lambda$ does not contain diagonal nodes. That is, $U_\lambda$ consists only on nodes strictly over the diagonal, so that the last row containing nodes from $U_\lambda$ is row $2$, since $k(\tilde{\lambda})=2$.

Since $m=0$, then every $3$-segment of $U_\lambda$ has $3$ nodes. The bottom $3$-segment of $U_\lambda$ is shown in shaded nodes in the following diagram

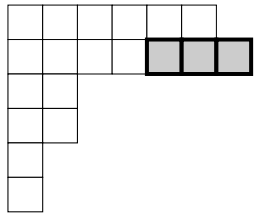
\[
\gyoung(;;;;;;,;;;;!\lt!\fg;;;!\lno!\fw,;;,;;,;,;)
\]

Since we have not reached the top of $[\tilde{\lambda}]$, there is at least another $3$-segment, which starts at the following upper row :
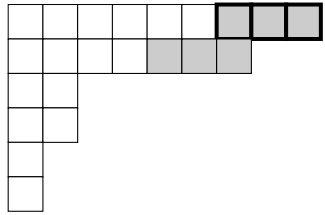
\[
\gyoung(;;;;;;!\lt!\fg;;;!\fw!\lno,;;;;!\fg;;;!\fw,;;,;;,;,;)
\]
Now, $\lambda$ is self-conjugate, then for each upper node $(i,j)$ that we added, we add the node $(j,i)$ (or also because $(i,j)\in U_\lambda$ if and only if $(j,i)\in L$)

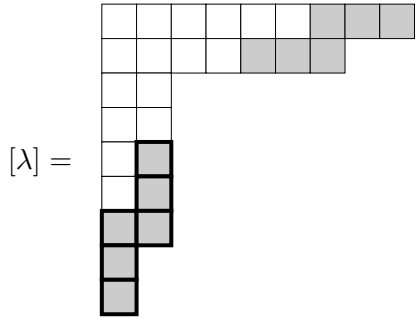
\[[\lambda]=\ \ 
 \gyoung(;;;;;;!\fg;;;!\fw,;;;;!\fg;;;!\fw,;;,;;,;!\lt!\fg;!\lno!\fw,;!\lt!\fg;,;;,;,;)
\]

 So that $\lambda=(9,7,2^5,1^2)$ is the only possibility for $\lambda$ self-conjugate such that $a^*_\lambda$ is even, 
$r_\lambda^*-\varepsilon_\lambda^*= \# U_\lambda \equiv 0\ (\textup{mod}\ 3)$ and 
 $\lambda^{(1)*}=\tilde{\lambda}$.
 
\item[-] Let $\varepsilon=1$, and $m=2$. Let us see that there is only one possible partition $\lambda$ satisfying: $a_\lambda^*$ is odd, 
$r_\lambda^*-\varepsilon_\lambda^* =\#\,U_\lambda-\varepsilon_\lambda^* \equiv 2\ (\textup{mod}\ 3)$ and 
 $\lambda^{(1)*}=\tilde{\lambda}$. We add to $[\tilde{\lambda}]$ the nodes of $\text{Rim}_3^*(\lambda)=U_\lambda\cup L_\lambda$.
 
In this case, $a^*_\lambda \equiv 1\ (\text{mod}\ 2)$, then there is a diagonal node in $U_\lambda$:
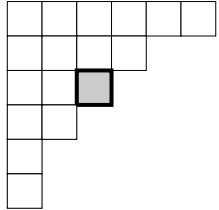
\[
\gyoung(;;;;;;,;;;;,;;!\lt!\fg;!\lno!\fw,;;,;,;)
\]
We add now the rest of the nodes in $U_\lambda$. Here $r_\lambda^*-\varepsilon_\lambda^* =\#(U_\lambda \setminus \{(3,3)\}) \equiv 2\ (\textup{mod}\ 3)$. That means that the rest of the nodes in $U_\lambda$ contain one $3$-segment of $2$ nodes, we add this $3$-segment
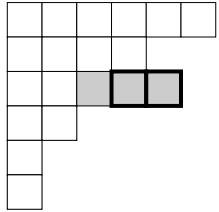
\[
\gyoung(;;;;;;,;;;;,;;!\fg;!\lt;;!\lno!\fw,;;,;,;)
\]
and the rest are $3$-segments of $3$ nodes:
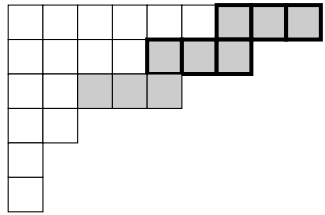
\[
\gyoung(;;;;;;!\lt!\fg;;;!\lno!\fw,;;;;!\lt!\fg;;;!\lno!\fw,;;!\fg;;;!\fw,;;,;,;)
\]
and finally, for making $\lambda$ self-conjugate:
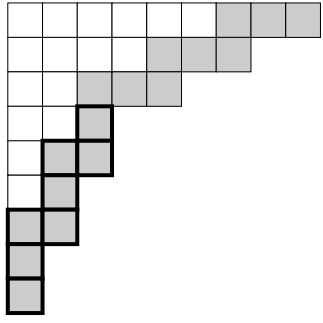
\[
\gyoung(;;;;;;!\fg;;;!\fw,;;;;!\fg;;;!\fw,;;!\fg;;;!\fw,;;!\lt!\fg;!\lno!\fw,;!\lt!\fg;;!\fw,!\lno;!\lt!\fg;,;;,;,;)
\]
And we see that $\lambda=(9,7,5,3^2,2^2,1^2)$ is the only possibility for having $a^*_\lambda$ odd, 
$r_\lambda^*-\varepsilon_\lambda^*= \#U_\lambda -1\equiv 2\ (\textup{mod}\ 3)$, and 
 $\lambda^{(1)*}=\tilde{\lambda}$.
 \end{itemize}
\end{example}
\medskip

\begin{proposition}\label{injective}
Let $p$ be an odd prime. Two different self-conjugate partitions have different BG-symbols. In other words, the BG-symbol gives rise to an injective map from self-conjugate partitions to the set of two-row positive integer symbols.
\end{proposition}

\begin{proof}
We proceed by induction on $l$, the length of the BG-symbol. Let $l=0$. Let $\lambda$ be a self-conjugate partition with BG-symbol
\[\bgs{p}(\lambda)=\begin{pmatrix}
a^*_0 \\
r^*_0\\
\end{pmatrix}.
\]
Notice that having a BG-symbol of length $0$ means that $\lambda=\lambda^{(0)*}=\lambda^{(l)*}$ is a hook and its size is $a^*_0$, that is, there are positive integers $u,v$ such that $
\lambda=(u,1^v)\ \ \text{and}\ \ u+v=a^*_0
$. But $\lambda$ is self-conjugate, then $u-1=v$, so that $a^*_0=2u-1$, and $\varepsilon_0^*=1$. Then 
\[
r^*_0=\frac{a_0^*+\varepsilon_0^*}{2} = u.
\]
Therefore $\lambda=(r^*_0,1^{r^*_0-1})$. This way, $\lambda$ is determined from its BG-symbol and, from this reasoning, we see that any self-conjugate partition with BG-symbol $\bgs{p}(\lambda)$ is completely determined by it and is then equal to $\lambda$. 

Now, fix $l>0 \in\mathbb{N}$, let $\lambda$ be a self-conjugate partition with BG-symbol
\[
\bgs{p}(\lambda)=
\begin{pmatrix}
a^*_0 & a^*_1 & \cdots & a^*_l \\
r^*_0 & r^*_1 & \cdots & r^*_l
\end{pmatrix},
\]
and let $\mu$ be a self-conjugate partition such that $\bgs{p}(\mu)=\bgs{p}(\lambda)$.

Then, by definition, the BG-symbol of $\lambda^{(1)*}$ (and also of $\mu^{(1)*}$) is the BG-symbol of $\lambda$ after removing its first column
\[
\bgs{p}(\lambda^{(1)*})=\bgs{p}(\mu^{(1)*})=
\begin{pmatrix}
 a^*_1 & \cdots & a^*_l \\
 r^*_1 & \cdots & r^*_l
\end{pmatrix}.
\]
By induction, there exists a unique self-conjugate partition $\tilde{\tau}$ such that
\[
\bgs{p}(\tilde{\tau})=
\begin{pmatrix}
 a^*_1 & \cdots & a^*_l \\
 r^*_1 & \cdots & r^*_l
\end{pmatrix}.
\]
Then $\tilde{\tau}=\lambda^{(1)*}=\mu^{(1)*}$. Let us see, from Lemma \ref{lema1}, that $\lambda=\mu$.

Let $\varepsilon= a^*_0\ \textup{mod}\ 2$, and $m= r^*_0-\varepsilon^*_0\ \textup{mod}\ p$. By Lemma \ref{parite}, if $\varepsilon=0$ then $m=0$. Therefore, by Lemma \ref{lema1} there exists a unique self-conjugate partition $\tau$ such that
\begin{itemize}
\item[(i)] $a^*_\tau \equiv \varepsilon\ (\textup{mod}\ 2)$;
\item[(ii)] $r^*_\tau-\varepsilon^*_\tau \equiv m\ (\textup{mod}\ p)$ and 
\item[(iii)] $\tau^{(1)*}=\tilde{\tau}$.
\end{itemize}
But partitions $\lambda$ and $\mu$ are self-conjugate and they satisfy (i) and (ii) since $a^*_0=a^*_\lambda=a^*_\mu$. Moreover $\lambda^{(1)*}=\mu^{(1)*}=\tilde{\tau}$, then by the uniqueness of $\tau$ we have that $\tau=\lambda=\mu$.

\end{proof}

As it turns out, the BG-symbol of a BG-partition is a Mullineux symbol of some self-Mullineux partition. Denote by $\mathcal{M}_p$ the set of Mullineux symbols of the self-Mullineux partitions $M_p$.

\begin{proposition}\label{prop:imagesymbol} Let $p$ be an odd prime and $\lambda$ a BG-partition. The BG-symbol of $\lambda$, $\bgs{p}{}(\lambda)$ is the Mullineux symbol of some self-Mullineux partition. That is
\[
\bgs{p}(\bg{p}{}) \subseteq \mathcal{M}_p.
\]
\end{proposition}

We postpone the proof of this proposition to Section \ref{subsec:imagesymb}, since some technical lemmas are necessary for this proof.

Recall, from Proposition \ref{cinco} and Remark \ref{RemConst}, that to a Mullineux symbol corresponds a unique $p$-regular partition. So that the Mullineux symbol determines a bijection between $p$-regular partitions and their Mullineux symbols. In particular, to a symbol in $\mathcal{M}_p$ corresponds a unique self-Mullineux partition in $M_p$. As a corollary from Proposition \ref{injective}, Proposition \ref{prop:imagesymbol} and from the fact that the sets $\bg{p}{n}$ and $M_p^n$ have the same number of elements, we obtain the following result.

\begin{theorem}\label{th:main} 
We have that
\[
\bgs{p}(\bg{p}{}) = \mathcal{M}_p,
\]
and the BG-symbol provides an explicit bijection between BG-partitions and self-Mullineux partitions. This bijection is given by associating to a BG-partition $\lambda$ its BG-symbol $\bgs{p}{}(\lambda)$, which corresponds to a unique self-Mullineux partition.
This bijection restricts to bijections between $\bg{p}{n}$ and $M^n_p$ for every $n\in\mathbb{N}$.
\end{theorem}

\begin{remark} If we consider the Mullineux symbol $G_p$ as a bijection from the set of $p$-regular partitions into its image in the set of two-row arrays of integers, then the bijection in Theorem \ref{th:main}, from $\bg{p}{}$ to $M_p$ is given precisely by $G_p^{-1}\circ\bgs{p}{}$.
\end{remark}

\begin{remark} \label{rem:bess}
In \cite{andrewsolsson}, G. Andrews and J. Olsson prove a general partition identity, which depends on some integer parameters. A special case of this identity is the fact that the number of self-Mullineux partitions of a non-negative integer $n$ equals the number of partitions of $n$ with different odd parts, none of them divisible by $p$, which is in turn equal to the number of BG-partitions. 

Now, in \cite{bessenrodt}, C. Bessenrodt shows a combinatorial proof of the Andrews--Olsson identity, which provides, by choosing the right parameters, an explicit bijection between $\bg{p}{n}$ and $M_p^n$. The bijection from Theorem \ref{th:main} is obtained in a more direct way and it is different from Bessenrodt's bijection. In particular, for $p=5$ and $n=20$, the partition $(7,6,3,2^2)\in M_5^{20}$ is mapped to partition $(9,3,2,1^6)\in \bg{5}{20}$ under Bessenrodt's bijection, and it is mapped to $(7,5,2^3,1^2)\in \bg{5}{20}$ under bijection from  Theorem \ref{th:main}.

\end{remark}

\subsection{Proof of the main result}\label{subsec:imagesymb}

In this subsection we prove Proposition \ref{prop:imagesymbol}. For this proof we need some technical lemmas. 

\begin{lemma}
\label{lemma:supregular} Let $p$ be an odd prime and $\lambda=(\lambda_1,\lambda_2,\ldots,\lambda_r)$ a partition in $\bg{p}{}$. Let $k=k(\lambda)$ as defined in Section \ref{secmulli}. Then the partition $\mu=(\lambda_1,\lambda_2,\ldots,\lambda_k)$ is $p$-regular.

\end{lemma}

\begin{example} Let $p=3$ and $\lambda=(7,4,3,2,1^3)$. This partition is self-conjugate and does not have diagonal $(3)$-hooks, that is, $\lambda \in \bg{3}{}$. Here $k=3$, so that $l(\mu)=3$, and indeed, the partition $\mu=(7,4,3)$ is $3$-regular. The following diagram illustrates both the partitions $\lambda$ and $\mu$ (in shaded boxes).
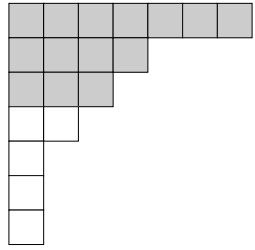
\[
\gyoung(!\fg;;;;;;;,;;;;,;;;,!\fw;;,;,;,;)
\]
\end{example}
\medskip
\begin{proof}[Proof of the lemma]
Suppose that $\mu$ is not $p$-regular. So that there exists $1 \leq i \leq k$ such that
\[
\lambda_{i}=\lambda_{i+1}=\cdots=\lambda_{i+p-1}.
\]
Since $i+p-1 \leq k$, then $\lambda_{i+p-1} \geq i+p-1$. Then $(i+p-1,i+p-1)\in [\mu] \subseteq [\lambda]$. Let $a$ be the length of the $(i+p-1,i+p-1)$-th hook of $\lambda$,  that is $a=h^\lambda_{(i+p-1,i+p-1)}$. 
Then, the length of the $(i+p-2,i+p-2)$-th hook of $\lambda$,   is $h_{(i+p-2,i+p-2)}=a+2$, since $\lambda_{i+p-1}=\lambda_{i+p-2}$ and $\lambda$ is self-conjugate. And $h_{(i+p-3,i+p-3)}=a+4$. In general $h_{(i+p-1-j,i+p-1-j)}=a+2j$ for $j=0,\ldots,p-1.$ That is, the lengths of these hooks are:
\[a,\,a+2,\,a+4,\,\ldots,\,a+2j,\,\ldots,\,a+2(p-1).
\]
But since $p\neq 2$, this list, modulo $p$, forms a complete collection of residues. Then, there exists $j\in \{0,\ldots,p-1\}$ such that $p \mid a+2j=h_{(i+p-1-j,i+p-1-j)}$, and this contradicts the fact that $\lambda \in \bg{p}{}.$

\end{proof}
In the set of BG-partitions, the implication in Lemma \ref{parite} becomes an equivalence:

\begin{lemma}
\label{lemmaepsilon}
Let $\lambda \in \bg{p}{}$. Then $a_{\lambda}^*$ is even if and only if $p \mid a^*_\lambda$.
\end{lemma}

\begin{proof}
As already noted, the fact that $a_{\lambda}^*$ implies $p \mid a^*_\lambda$ is proved in Lemma \ref{parite}.

Suppose that $p \mid a^*_\lambda$. If $a^*_\lambda$ is odd, then $\textup{Rim}^*_p(\lambda)$ contains a diagonal node. Then $U_\lambda$ is formed by $p$-segments of length $p$ and one last $p$-segment of length possibly less than $p$, which, in this case contains the diagonal node. Let us name $B$ the set of nodes in this last $p$-segment, and let $A$ be the set 
\[
A=B \cup \{(j,i) \in [\lambda] \mid (i,j)\in B\} \subseteq \textup{Rim}^*_p(\lambda).
\]
 The set $A$ is formed by a symmetrical segment along the rim of $[\lambda]$. See Figure \ref{centr}. 
 
 \begin{figure}[h]
  \begin{subfigure}[b]{0.4\textwidth}
  \[
\gyoung(:::::::::::j,,::;_8\hdts!\fg;!\fw,::|8\vdts:,,:i:::::::::!\fg a,::::::::::|1\vdts,:::!\fw::::!\fg;_2\hdts<>,:::::::|2\vdts:,::::::::::<A>,:::::;_1\hdts;,::!\fg;)
\]    
    \caption{Segment $A$.}
  
    \label{centr}
  \end{subfigure}
  \begin{subfigure}[b]{0.4\textwidth}
    \[
\gyoung(::::::i:::::j,,::;_8\hdts!\fg;!\fw,::|8\vdts:,,:i::::!\fgg;_4\hdts a,!\fgg:::::|4\vdts!\fg::::|1\vdts,:::!\fw::::!\fg;_2\hdts<>,:::::::|2\vdts:,::::::::::<A>,:::::!\fgg;!\fg_1\hdts;,::!\fg;)
\]
    \caption{$(i,i)_\lambda$-th hook in darker shaded boxes.}
	\label{cent}
  \end{subfigure}
  \caption{}
\end{figure}

The set $\textup{Rim}^*_p(\lambda)$ is formed by the disjoint union of $A$ and some $p$-segments. Therefore, since $p \mid a^*_\lambda$, we have that $|A|=p$.

\noindent Now, let $a=(i,j)$ be the first node of the segment $B$, that is $i=\textup{min} \{r \mid (r,s)\in B\}$ and $j=\textup{max} \{s \mid (r,s)\in B\}$. We have that the $(i,i)$-th hook contains exactly $|A|=p$ nodes. See Figure \ref{cent}.

This means that $\lambda$ has a diagonal $(p)$-hook, which is a contradiction because $\lambda \in \bg{p}{}$.

\end{proof}

\noindent We obtain the following corollary from Remark \ref{parite2} and Lemma \ref{lemmaepsilon}.

\begin{corollary}
\label{corollary:epsilon}
Let $\lambda \in \bg{p}{}$. The following statements are equivalent
\begin{enumerate}
\item $\varepsilon^*_\lambda = 0$.
\item $a^*_\lambda$ is an even integer.
\item $\textup{Rim}^*_p(\lambda)$ has no diagonal nodes.
\item $p \mid a^*_\lambda$.

\end{enumerate}
\end{corollary}

Consider a partition $\lambda \in M_p$, that is, a fixed point of the Mullineux map. This is a condition that depends only on the columns of the Mullineux symbol of $\lambda$.
Therefore, the partition $\lambda^{(1)}$ obtained by removing the $p$-rim of $\lambda$ is also a fixed point of the Mullineux map, since its Mullineux symbol is obtained by removing the first column of the Mullineux symbol of $\lambda$. The following lemma is an analogue property for partitions in $\bg{p}{}$.

\begin{lemma} 
\label{lemma:quitarrim*}
If $\lambda \in \bg{p}{}$ then $\lambda^{(1)*} \in \bg{p}{}$. In other words, if $\lambda$ is a BG-partition, then, removing its $p$-rim* results in a BG-partition.
\end{lemma}

\begin{proof}
Recall  (Remark \ref{removeprim}) that if $\lambda$ is self-conjugate, then so it is for $\lambda^{(1)*}$. In particular, if $\lambda \in \bg{p}{}$, then $\lambda^{(1)*}$ is self-conjugate. It remains to prove that $\lambda^{(1)*}$ does not have any diagonal $(p)$-hooks.

For simplicity of notations let $\mu=\lambda^{(1)*}$. Suppose that $\mu$ has a diagonal $(p)$-hook, say the $(i,i)_\mu$-th hook, with $h^\mu_{i,i}=pk$ for some $k \in \mathbb{N}$.

We claim that the node $(i,\mu_i+1)$ is in the $p$-rim* of $\lambda$. Indeed, $(i,\mu_i+1) \in [\lambda]$ since if $(i,\mu_i+1) \notin [\lambda]$, then $\mu_i=\lambda_i$ and $h^\lambda_{i,i}=h^\mu_{i,i}=pk$ so that $\lambda$ has a diagonal $(p)$-hook, which contradicts the fact that $\lambda\in\bg{p}{}$. Now, since $(i,\mu_i+1)\in [\lambda]\setminus [\mu]$, then $(i,\mu_i+1)\in \textup{Rim}_p^*(\lambda)$. See Figure \ref{figure:rim*bg}.

 \begin{figure}[h]
  \begin{subfigure}[b]{0.46\textwidth}
  
      \[
\gyoung(:::::::::::<\mu_i+1>,::::::i:::::\downarrow,,::;_9\hdts<>,::|9\vdts:,,:i::::!\lt;_3\hdts<>!\lno!\fg;!\fw,:::::!\lt|3\vdts,:::::;,:::::!\lno!\fg;!\fw,,::;)
\]
\caption{Shaded boxes are in $\textup{Rim}_p^*(\lambda)$.}
\label{figure:rim*bg}

  \end{subfigure}
  \begin{subfigure}[b]{0.46\textwidth}

    \[
\gyoung(::::::::::::<\mu_i+1>,:::::a:::i::::\downarrow::b,,::;_\twelve\hdts;,::|\twelve\vdts:,:a:::;_8\hdts!\fg;!\fw,::::|8\vdts::::::::!\fg|2\vdts:!\fw,,:i::::::!\lt;_2\hdts<>!\lno!\fg;_1\hdts;!\fw,:::::::!\lt|2\vdts,:::::::;!\lno,:::::::!\fg;!\fw,:::::::!\fg|1\vdts:,::::;_2\hdts<>,,::!\fw;)
\]
\caption{Shaded boxes are $p$-segments in $\textup{Rim}_p^*(\lambda)$.}
\label{figure:rim*bgsegments}

  \end{subfigure}
  \caption{}
\end{figure}

There are now two possible cases: either $(i,\mu_i+1)$ is the last node of a $p$-segment of $U_\lambda$ (the nodes on the $p$-rim* of $\lambda$ over the diagonal), or it is not the last node of the $p$-segment to which it belongs. Let us examine these two cases.

Suppose $(i,\mu_i+1)$ is the last node of a $p$-segment of $U_\lambda$, and this $p$-segment starts on a node $(a,b)$. See Figure \ref{figure:rim*bgsegments}.

Then, the $(a,a)_\lambda$-th hook has length equal to the length of the $(i,i)_\mu$-th hook plus twice the length of the $p$-segment of $\textup{Rim}_p^*(\lambda)$ containing the node $(i,\mu_i+1)$, that is 
\[h_{a,a}^\lambda=p+h_{i,i}^\mu+p=p+pk+p=p(k+2),\]

\noindent so that $\lambda$ contains diagonal $(p)$-hook, which is impossible.

Suppose now that $(i,\mu_i+1)$ is not the last node of a $p$-segment of $U_\lambda$. First, notice that the node $(i+1,\mu_i+2) \notin [\lambda]$. This is true because $(i,\mu_i+1)$ is in the $p$-rim* of $\lambda$. We claim that $(i+1,\mu_i+1)\in\textup{Rim}_p^*(\lambda) \subseteq [\lambda]$. In Figure \ref{fig:rimcroch}, node $(i+1,\mu_i+2)$ is illustrated as a cross (meaning it is not in $[\lambda]$) and node $(i+1,\mu_i+1)$ as a shaded box (as are their opposites with respect to the diagonal). 
\begin{figure}[h]
      \[
\gyoung(:::::::::::<\mu_i+1>,::::::i:::::\downarrow,,::;_9\hdts<>,::|9\vdts:,,:i::::!\lt;_3\hdts<>!\lno!\fg;!\fw,:<i+1>::::!\lt|3\vdts::::!\lno!\fg;!\fw:\equis !\lt,,,:::::;,:::::!\lno!\fg;;!\fw,:::::::\equis,::;)
\]
\caption{}
\label{fig:rimcroch}
\end{figure}
Indeed, $(i+1,\mu_i+1)\in\textup{Rim}_p^*(\lambda)$ because, since $(i,\mu_i+1)$ is not the last node of a $p$-segment, then the next node of its $p$-segment is either to the left or down. But the node to the left of $(i,\mu_i+1)$, that is, $(i,\mu_i)$ is not in the $p$-rim* of $\lambda$ since it is in $\mu$, so that the next node of this $p$-segment is $(i+1,\mu_i+1)$, which is then in $\textup{Rim}_p^*(\lambda)$.

The fact that $(i+1,\mu_i+1)\in\textup{Rim}_p^*(\lambda) \subseteq [\lambda]$ and $(i+1,\mu_i+2) \notin [\lambda]$ implies that $\lambda_{i+1}=\mu_{i}+1$ and therefore the $(i+1,i+1)_\lambda$-th hook has length
\[ h^\lambda_{(i+1,i+1)}=h^\mu_{(i,i)}=pk,
\]
that is, $\lambda$ has a diagonal $(p)$-hook, a contradiction.

We  conclude that $\mu$ does not have any diagonal $(p)$-hooks and then, $\mu=\lambda^{(1)*} \in \bg{p}{}.$
\end{proof}

\begin{proof}[Proof of Proposition \ref{prop:imagesymbol}]
Let us first state which properties characterize elements in $\mathcal{M}_p$. That is, if $\lambda \in M_p^n$ which conditions characterize its Mullineux symbol 
\[
G_p(\lambda)=
\begin{pmatrix}
a_0 & a_1 & \cdots & a_l \\
r_0 & r_1 & \cdots & r_l
\end{pmatrix}.
\]
Let $\varepsilon_i$ be as in Proposition \ref{cinco} and $s_i=a_i+\varepsilon_i-r_i$. The partition $\lambda$ is a fixed point of the Mullineux map if and only if $r_i=s_i$, that is
\[
a_i=2\,r_i-\varepsilon_i.
\]
We also know that $\lambda$ is the only $p$-regular partition whose Mullineux symbol satisfies properties (1)--(5) from Proposition \ref{cinco}. This way, the properties that characterize Mullineux symbols of partitions in $M_p$ are equivalent to the following properties

\begin{enumerate}
\item $\varepsilon_i \leq r_i-r_{i+1} < p+\varepsilon_i $ for $0 \leq i < l$,
\item $1 \leq r_l < p+\varepsilon_l$, 
\item $\sum_{i=0}^l a_i =n$, and
\item $a_i=2r_i-\varepsilon_i$.
\end{enumerate}

On the other hand, from the definition of $\varepsilon_i^*$ and Corollary \ref{corollary:epsilon}, we have that 

\[ \varepsilon^*_i =
\begin{cases}
      0 \ \text{if}\  p \mid a^*_i \\
      1 \ \text{if} \ p \nmid a^*_i 
   \end{cases}
\]

Let $\lambda \in \bg{p}{n}$. Let us see that its BG-symbol 
\[\bgs{p}(\lambda)=
\begin{pmatrix}
a^*_0 & a^*_1 & \cdots & a^*_l \\
r^*_0 & r^*_1 & \cdots & r^*_l
\end{pmatrix}
\] is in $\mathcal{M}_p$ by verifying properties (1)--(4) for $a_i^*$, $\varepsilon^*_i$ and $r^*_i$ : 

From the definition of the sequence $a^*_0,\ldots,a^*_l$, it is clear that (3) holds. We have that (4) is satisfied from Remark \ref{parite2}. Let us first show that (2) holds.
Since $\lambda^{(l)*}$ is not the empty partition, $r^*_l \geq 1$. On the other hand, the partition $\lambda^{(l)*}$ is a hook and it is self-conjugate; more precisely $\lambda^{(l)*}=(r^*_l,1^{a^*_l-r^*_l})$. Then $a^*_l=|\lambda^{(l)*}|$ is odd, so that $\varepsilon^*_l=a^*_l\ \textup{mod} \ 2=1$. Suppose that $r^*_l \geq p+\varepsilon^*_l=p+1$. This means that the first $p$-segment over the diagonal of $\lambda^{(l)*}$ consists of $p$ nodes. But then, there are more nodes remaining in the first row of $[\lambda^{(l)*}]$ that are not in the $p$-rim*, but this contradicts the maximality of $l$. 

 It remains to prove (1). A key element for this task is Lemma \ref{lemma:supregular}, which roughly says that truncating a BG-partition to some particular row results in a $p$-regular partition. The idea is to use the fact that this truncated partition, being $p$-regular, satisfies properties from Proposition \ref{cinco}, which uses numbers from the Mullineux symbol, and these will give us information about $r^*_i$ and $\varepsilon^*_i$, which are numbers appearing in the BG-symbol.

Let us see that $\lambda$ satisfies
\[\varepsilon^*_i \leq r^*_i-r^*_{i+1} < p+\varepsilon^*_i  \ \ \text{for}\ \  0 \leq i < l.
\]

It suffices to prove this for $i=0$ and then, the property follows recursively by Lemma \ref{lemma:quitarrim*}.

\noindent To simplify notation, set: 

\begin{center}
\begin{tabular}{c|c} 

 values associated to $\lambda$  & values associated to $\lambda^{(1)*}$ \\ 
 \hline
  $a:= a^*_0$ &  $a':= a^*_1$ \\ 
 $r:=r^*_0$ &  $r':=r^*_1$\\ 
  $\varepsilon:=\varepsilon^*_0$ & $\varepsilon':=\varepsilon^*_1$ \\ 

\end{tabular}
\end{center}

\noindent Let us prove that 
\[
\varepsilon \leq r-r' < p+\varepsilon.
\]
We study the four possible cases for the values of $\varepsilon$ and $\varepsilon'$, namely

\begin{center}
\begin{tabular}{ccc} 

  & $\varepsilon$ & $\varepsilon'$ \\ 
 \hline
 (i) & $0$ & $0$ \\ 
 (ii) & $0$ & $1$ \\ 
  (ii) & $1$ & $0$ \\ 
   (iv) & $1$ & $1$ \\ 

\end{tabular}
\end{center}

In each case we will consider some diagram
\[
[\underline{\lambda}]:=\{(i,j)\in [\lambda] \mid i\leq k(\lambda) \text{ and } j\geq k(\lambda)-x+1\},
\]
for a certain $1\leq x\leq k(\lambda)$ (which will be chosen depending on the case). That is, $[\underline{\lambda}]$ is obtained from $[\lambda]$ by removing rows below row $k(\lambda)$ and columns up to column $k(\lambda)-x$. In an abuse of notation we will call $\underline{\lambda}$ the partition with Young diagram obtained by shifting the diagram $[\underline{\lambda}]$ to the left by $k(\lambda)-x$ columns. This will allow us to identify nodes of $\underline{\lambda}$ and nodes of $\lambda$ (for example $(i,\lambda_i)$, and not $(i,\lambda_i-k(\lambda)+x)$, will be the last node on row $i$ of $\underline{\lambda}$ for $1\leq i \leq k(\lambda)$).

In each case we denote $\underline{a}$ the number of nodes in $\textup{Rim}_p(\underline{\lambda})$, $\underline{r}$ the length of $\underline{\lambda}$ and 
\[ \underline{\varepsilon} =
\begin{cases}
      0 \ \text{if}\  p \mid \underline{a}, \\
      1 \ \text{if} \ p \nmid \underline{a}.
   \end{cases}
\]
 And for $\underline{\lambda}^{(1)}$, similar notation with primes: $\underline{a}'$ the number of nodes in $\textup{Rim}_p(\underline{\lambda}^{(1)})$, $\underline{r}'$ the length of $\underline{\lambda}^{(1)}$ and 
\[ \underline{\varepsilon'} =
\begin{cases}
      0 \ \text{if}\  p \mid \underline{a}', \\
      1 \ \text{if} \ p \nmid \underline{a}'.
   \end{cases}
\]

Let us now consider each of the four cases.
\medskip

\begin{enumerate}
\item[(i)] In this case, the fact that both $\varepsilon$ and $\varepsilon'$ are zero means that neither $\lambda$ nor $\lambda^{(1)*}$ contain diagonal nodes on their $p$-rims*. For example as in the partition $(6,5,2^3,1)$ with $p=3$. In Figure \ref{fig:ex1rim} we represent the $3$-rims* of $(6,5,2^3,1)$ and $(6,5,2^3,1)^{(1)*}$ in different shades.

\begin{figure}[h]
  \begin{subfigure}[b]{0.48\textwidth}
\[
\gyoung(;;;!\fgg;!\fg;;!\fw,;;!\fgg;;!\fg;!\fw,;!\fgg;,;;,!\fg;;,;)
\]
\caption{$\textup{Rim}^*_3((6,5,2^3,1))$ and $\textup{Rim}^*_3((6,5,2^3,1)^{(1)*}).$}
\label{fig:ex1rim}
 
  \end{subfigure}
  \begin{subfigure}[b]{0.48\textwidth}
\[
\gyoung(;!\lt;;!\fgg;!\fg;;!\lno!\fw,;!\lt;!\fgg;;!\fg;!\lno!\fw,;!\fgg;,;;,!\fg;;,;)
\]
\caption{Partition $\underline{(6,5,2^3,1)}$ in thicker lines.}
\label{fig:ex1rimsub}
	\end{subfigure}
	\caption{}
\end{figure}

Let $x=1$ in this case. Figure \ref{fig:ex1rimsub} illustrates $\underline{(6,5,2^3,1)}$. Lemma \ref{lemma:supregular} ensures that $\underline{\lambda}$ is $p$-regular, then, from Proposition \ref{cinco}, we have 
\begin{equation}
\label{eq:1cinco3}
\underline{r}-\underline{r}'+ \underline{\varepsilon}' \leq \underline{a}-\underline{a}' < p+ \underline{r}-\underline{r}'+\underline{\varepsilon}'.
\end{equation}
Notice that the nodes in $\textup{Rim}_p^*(\lambda)$ over the diagonal of $\lambda$ are exactly the nodes of $\textup{Rim}_p(\underline{\lambda})$, that is, $U_\lambda=\textup{Rim}_p(\underline{\lambda})$. Hence $\#\,U_\lambda=\#\,\textup{Rim}_p(\underline{\lambda})$. That is $r=\underline{a}$. Similarly, $U_{\lambda^{(1)*}}=\textup{Rim}_p(\underline{\lambda}^{(1)})$, since $\varepsilon'=0$, meaning that node $(k(\lambda)-1,k(\lambda)-1) \notin \textup{Rim}_p^*(\lambda^{(1)*})$ so that this node is not in $\textup{Rim}_p(\underline{\lambda}^1)$, either. Hence $r'=\underline{a}'$. 

We claim that $\textup{Rim}_p(\underline{\lambda})$ and $\textup{Rim}_p(\underline{\lambda}^{(1)})$ end at the same row; row $k(\lambda)$. This is not obvious since it could be possible that the $p$-rim of a partition $\mu$, which always contains nodes in the last row of $\mu$, row $l(\mu)$, contains every node in this last row, and then $\mu^{(1)}$ does not have any nodes in row $l(\mu)$. But in  our setting, this is not the case. Indeed, by definition, every node of a partition is in some $i$-th $p$-rim of the partition. In particular, the diagonal node $(k(\lambda),k(\lambda))$ is in the $j$-th $p$-rim of $\underline{\lambda}$ for some $j > 1$ since $\textup{Rim}_p(\underline{\lambda})$ and $\textup{Rim}_p(\underline{\lambda}^{(1)})$ do not contain diagonal nodes. On the other hand the $p$-rim of any partition contains nodes in the last row of the partition and since $k(\lambda)$ is the last row of both $\underline{\lambda}$ and $\underline{\lambda}^{(j)}$, then it is also the last row of $\underline{\lambda}^{(1)}$. So that both $\underline{\lambda}$ and $\underline{\lambda}^{(1)}$ contain nodes in row $k(\lambda)$.

Now, the fact that $\textup{Rim}_p(\underline{\lambda})$ and $\textup{Rim}_p(\underline{\lambda}^{(1)})$ end at the same row means that $l(\underline{\lambda})=k(\lambda)=l(\underline{\lambda}^{(1)})$, that is $\underline{r}-\underline{r}'=0$. 

On the other hand, in this case, we have that $\varepsilon'=0$, which means that $p\mid a'$. But since $a'=2r'-\varepsilon'$, then $p \mid r'=\underline{a}'$ ($p \neq 2$), which means that $\underline{\varepsilon}'=0.$

The fact that $\underline{r}-\underline{r}'=0$, together with the fact that $\underline{a}=r$, $\underline{a}'=r'$ and $\underline{\varepsilon}=0$, make Equation \ref{eq:1cinco3} become
\[
0 \leq r- r'<p+0.
\]
So that $\varepsilon \leq r- r'<p+\varepsilon$, as we wanted to show, since in this case $\varepsilon=0$.
\vspace{.5cm}

\item[(ii)] Suppose that $\varepsilon=0$ and $\varepsilon'=1$. For example as in the partition $(7,5,2^3,1^2)$ with $p=3$. In Figure \ref{fig:ex2rim} we represent the $3$-rims* of $(7,5,2^3,1^2)$ and $(7,5,2^3,1^2)^{(1)*}$ in different shades.

\begin{figure}[h]
  \begin{subfigure}[b]{0.48\textwidth}
\[
\gyoung(;!\fgg;;;!\fg;;;,!\fgg;;!\fg;;;,!\fgg;!\fg;,!\fgg;!\fg;,;;,;,;)
\]
\caption{$\textup{Rim}^*_3((7,5,2^3,1^2))$ and $\textup{Rim}^*_3((7,5,2^3,1^2)^{(1)*}).$}
\label{fig:ex2rim}
 
  \end{subfigure}
  \begin{subfigure}[b]{0.48\textwidth}
\[
\gyoung(;!\lt!\fgg;;;!\fg;;;,!\lno!\fgg;!\lt;!\fg;;;!\lno,!\fgg;!\fg;,!\fgg;!\fg;,;;,;,;)
\]
\caption{Partition $\underline{(7,5,2^3,1^2)}$ in thicker lines.}
\label{fig:ex2rimsub}
	\end{subfigure}
	\caption{}
\end{figure}

As in the previous case, let $x=1$. We illustrate $\underline{(7,5,2^3,1^2)}$ by thicker lines in Figure \ref{fig:ex2rimsub}.

Let us see that also in this case we have that $\underline{r}-\underline{r}'=0$. As before, the nodes in $\textup{Rim}_p^*(\lambda)$ over the diagonal of $\lambda$, or $U_\lambda$, are exactly the nodes of $\textup{Rim}_p(\underline{\lambda})$. And we also have that $U_{\lambda^{(1)*}}=\textup{Rim}_p(\underline{\lambda}^{(1)})$. So that $r=\underline{a}$ and $r'=\underline{a}'$. On the other hand, since in this case $(k(\lambda),k(\lambda))\in \textup{Rim}^*_p(\lambda^{(1)*})$, then $(k(\lambda),k(\lambda))\in \textup{Rim}_p(\underline{\lambda}^{(1)})$.
Furthermore, the fact that $\textup{Rim}^*_p(\lambda^{(1)*})$ has a node on row $k(\lambda)$, implies that $\underline{\lambda}^{(1)}$ has a node on row $k(\lambda)$, and then so it is for $\underline{\lambda}$. Therefore $l(\underline{\lambda})=k(\lambda)=l(\underline{\lambda}^{(1)})$, then
$\underline{r}-\underline{r}'=0$.

Now, consider the two possible cases for $\underline{\varepsilon}'$. If $\underline{\varepsilon}'=0$, we obtain, as in the previous case 
\[
0 \leq r- r'<p+0,
\]
which is what we wanted to show. If $\underline{\varepsilon}'=1$, Equation \ref{eq:1cinco3} becomes 
\[
1 \leq r-r' < p+1,
\]
In particular  $0 \leq r-r' \leq p$. But actually, $r-r'<p$. Indeed, if $r-r'=p$, since $p \mid a=2r$, then $p \mid r$ and therefore $p \mid r' = \underline{a}'$, which contradicts the fact that $\underline{\varepsilon}'=1$. In conclusion $0 \leq r-r' < p$, which ends this case.
\vspace{.5cm}

\item[(iii)] Suppose that $\varepsilon=1$ and $\varepsilon'=0$. For example as in the partition $(6,5^2,3^2,1)$ with $p=3$. In Figure \ref{fig:ex3rim} we represent the $3$-rims* of $(6,5^2,3^2,1)$ and $(6,5^2,3^2,1)^{(1)*}$ in different shades.

\begin{figure}[h]
  \begin{subfigure}[b]{0.48\textwidth}
\[
\gyoung(;;;!\fgg;!\fg;;,!\fw;;!\fgg;;!\fg;,!\fw;!\fgg;!\fg;;;,!\fgg;;!\fg;,;;;,;)
\]
\caption{$\textup{Rim}^*_3((6,5^2,3^2,1))$ and $\textup{Rim}^*_3((6,5^2,3^2,1)^{(1)*}).$}
\label{fig:ex3rim}
 
  \end{subfigure}
  \begin{subfigure}[b]{0.48\textwidth}
\[
\gyoung(;;!\lt;!\fgg;!\fg;;!\lno,!\fw;;!\lt!\fgg;;!\fg;!\lno,!\fw;!\fgg;!\fg!\lt;;;!\lno,!\fgg;;!\fg;,;;;,;)
\]
\caption{Partition $\underline{(6,5^2,3^2,1)}$ in thicker lines.}
\label{fig:ex3rimsub}
	\end{subfigure}
	\caption{}
\end{figure}

As before, let $x=1$. We illustrate $\underline{(6,5^2,3^2,1)}$ by thicker lines in Figure \ref{fig:ex3rimsub}. Let us see that in this case $\underline{r}-\underline{r}'=1$.

As in the preceding cases, the nodes in $\textup{Rim}_p^*(\lambda)$ over the diagonal of $\lambda$ are exactly the nodes of $\textup{Rim}_p(\underline{\lambda})$. This fact implies that $\underline{a}=r$, and since $\varepsilon'=0$, by the same argument that in case (i), we have that $\underline{a}'=r'$.  Now, since $\varepsilon=1$, the last diagonal node of $\lambda$, that is, the node $(k(\lambda),k(\lambda))$ is in $\textup{Rim}^*_p(\lambda)$. In particular $ (k(\lambda),k(\lambda))\in \textup{Rim}_p(\underline{\lambda})$, and since it is the first node of the last row of $\underline{\lambda}$, that means that all nodes on this last row are in the $p$-rim of $\underline{\lambda}$. So that this last row $k(\lambda)=\underline{r}$ of $\underline{\lambda}$ does not have any nodes from $\textup{Rim}_p^*(\lambda^{(1)*})$ (or $\textup{Rim}_p(\underline{\lambda}^{(1)})$). We claim that
row $k(\lambda)-1$ in $\underline{\lambda}$ contains at least one node in $\textup{Rim}_p^*(\lambda^{(1)*})$ (or $\textup{Rim}_p(\underline{\lambda}^{(1)})$). Indeed, node $(k(\lambda)-1,k(\lambda)-1)$ is in $\textup{Rim}_p^*(\lambda^{(j)*})$ for a $j>1$, since it is not in $\textup{Rim}_p^*(\lambda^{(1)*})$ (because $\varepsilon'=0$). If we suppose that row $k(\lambda)-1$ does not have node in $\textup{Rim}_p^*(\lambda^{(1)*})$, we are supposing that to the left of node $(k(\lambda)-1,k(\lambda)-1)$ there are only nodes from $\textup{Rim}_p^*(\lambda)$. If this is the case, the last node in row $k(\lambda)-1$ in $\lambda^{(1)*}$ is $(k(\lambda)-1,k(\lambda)-1)$, that is: $\lambda^{(1)*}_{k(\lambda)-1}=k(\lambda)-1$. But the last node on every row (over or on the diagonal) belongs to the $p$-rim*. In this case, node $(k(\lambda)-1,k(\lambda)-1)$ belongs to the $p$-rim* of $\lambda^{(1)*}$, a contradiction since $\lambda^{(1)*}$ does not have any diagonal nodes on its $p$-rim*. In conclusion, row $k(\lambda)-1$ in $\underline{\lambda}$ contains at least one node in $\textup{Rim}_p^*(\lambda^{(1)*})$, in particular, row $k(\lambda)-1$ in $\underline{\lambda}$ contains at least one node in $\underline{\lambda}^{(1)}$, so that $l(\underline{\lambda}^{(1)})=k(\lambda)$. Therefore $\underline{r}-\underline{r'}=k(\lambda)-l(\underline{\lambda}^{(1)})=1$.

On the other hand, since we have that $\varepsilon'=0$, by the same argument as in case (i), we have that $\underline{\varepsilon'}=0$.

Puting all together in Equation \ref{eq:1cinco3}, we obtain
\[
1 \leq r-r' < p+ 1.
\]
That is, $\varepsilon \leq r-r' < p+ \varepsilon$.
\vspace{.5cm}

\item[(iv)] Suppose finally that $\varepsilon=\varepsilon'=1$. An example is given by partition $(7,4,3,2,1^3)$ for $p=3$. In Figure \ref{fig:ex4rim} we represent the $3$-rims* of $(7,4,3,2,1^3)$ and $(7,4,3,2,1^3)^{(1)*}$ in different shades.

Let $x=2$. Lemma \ref{lemma:supregular} still assures that $\underline{\lambda}$ is $p$-regular. And from the way that it is defined, $\underline{\lambda}$ contains both diagonal nodes in $\textup{Rim}_p^*(\lambda)$ and $\textup{Rim}_p^*(\lambda^{(1)*})$. We illustrate $\underline{(7,4,3,2,1^3)}$ by thicker lines in Figure \ref{fig:ex4rimsub}.

\begin{figure}[h]
  \begin{subfigure}[b]{0.48\textwidth}
\[
\gyoung(;!\fgg;;;!\fg;;;,!\fgg;;!\fg;;,!\fgg;!\fg;;,!\fgg;!\fg;,;,;,;)
\]
\caption{$\textup{Rim}^*_3((7,4,3,2,1^3))$ and $\textup{Rim}^*_3((7,4,3,2,1^3)^{(1)*}).$}
\label{fig:ex4rim}
 
  \end{subfigure}
  \begin{subfigure}[b]{0.48\textwidth}
\[
\gyoung(;!\lt!\fgg;;;!\fg;;;!\lno,!\fgg;!\lt;!\fg;;!\lno,!\fgg;!\lt!\fg;;!\lno,!\fgg;!\fg;,;,;,;)
\]
\caption{Partition $\underline{(7,4,3,2,1^3)}$ in thicker lines.}
\label{fig:ex4rimsub}
	\end{subfigure}
	\caption{}
\end{figure}

Notice that in this case it is not necessarily true that $\underline{a}=r$ and $\underline{a}'=r'$. Since $\underline{\lambda}$ contains the node $(k(\lambda),k(\lambda)-1)$ which is under the diagonal of $\lambda$, where the $p$-rim* does not behave as the $p$-rim. For the partition $(7,4,3,2,1^3)$, this node is the node $(3,2)$, which in this case is in the $p$-rim* of $(7,4,3,2,1^3)$. But it could be the case that the node $(k(\lambda),k(\lambda)-1)$ is not in the $p$-rim* of $\lambda$ but in the $p$-rim* of $\lambda^{(1)*}$. This depends on the divisibility of $r$ by $p$.

Recall that $r=\#\,U_\lambda$ is the number of nodes in the $p$-rim* of $\lambda$ that are above (or on) the diagonal of $\lambda$. Let us consider the two cases: $p \mid r$ and $p \nmid r$.

\begin{itemize}
\item Suppose that $p \mid r$. As in $\lambda=(7,4,3,2,1^3)$ with $p=3$, see Figure \ref{fig:ex4rim}. In this case every $p$-segment of the $p$-rim* of $\lambda$ contains exactly $p$ nodes. In particular the segment which contains the node $(k(\lambda),k(\lambda))$. And since this node is the last (and $p$-th) node of this $p$-segment, then the node to its left $ (k(\lambda),k(\lambda)-1)$ is not in the $p$-rim of $\underline{\lambda}$. And we have $\underline{a}=r$ and $\underline{r}-\underline{r}'=0$. Moreover, the node $ (k(\lambda),k(\lambda)-1)$ is then in the $p$-rim of $\underline{\lambda}^{(1)}$. So that $\underline{a}'=r'+1$ (the $r'$ nodes of $\textup{Rim}_p^*(\lambda^{(1)*})$ above the diagonal, together with node $ (k(\lambda),k(\lambda)-1)$, form $\textup{Rim}_p(\underline{\lambda}^{(1)})$). Therefore we have $\underline{a}=r$, $\underline{a}'=r'+1$ and $\underline{r}-\underline{r}'=0$. Puting this together in Equation \ref{eq:1cinco3} we get
\[
 \underline{\varepsilon}' \leq r-(r'+1) < p+\underline{\varepsilon}',
\]
or
\[
 \underline{\varepsilon}'+1 \leq r-r' < p+\underline{\varepsilon}'+1.
\]
But $\underline{\varepsilon}'+1 \geq 1 = \varepsilon$. Therefore we have 
\[
1 \leq r-r' < p+\underline{\varepsilon}'+1.
\]
Let us see that $r-r' < p+1=p+\varepsilon$. There are two possibilities for $\underline{\varepsilon}'$. Either $\underline{\varepsilon}'=0$, in which case $r-r' < p+1$, or  $\underline{\varepsilon}'=1$. If $\underline{\varepsilon}'=1$, we have $r-r' < p+2$, so that $r-r' \leq p+1$. But in fact $r-r' < p+1$, since if $r-r' = p+1$, then $r-(r'+1) = p$. But in this case $p\mid r$, therefore, $p \mid r'+1=\underline{a}'$, and this contradicts the fact that $\underline{\varepsilon}'=1$.
\medskip

\item Suppose that $p \nmid r$. As in $\lambda=(6,2,1^4)$ with $p=3$, see Figure \ref{fig:ex5rim}. In this case, the $p$-segment of $\textup{Rim}_p^*(\lambda)$ which contains the node $(k(\lambda),k(\lambda))$ has less than $p$ nodes. This implies that the node to the left of this diagonal node, namely $(k(\lambda),k(\lambda)-1)$, is on the $p$-rim of $\underline{\lambda}$. Then $\underline{a}=r+1$ (the $r$ nodes of $\textup{Rim}_p^*(\lambda)$ above the diagonal, together with node $ (k(\lambda),k(\lambda)-1)$, form $\textup{Rim}_p(\underline{\lambda})$), and we also have that $\underline{r}-\underline{r}'=1$ and $\underline{a}'=r'$. Equation \ref{eq:1cinco3} gives
\[
1+ \underline{\varepsilon}' \leq (r+1)-r' < p+ 1+\underline{\varepsilon}',
\]
or
\[
\underline{\varepsilon}' \leq r-r' < p+\underline{\varepsilon}'.
\]
But $p+\underline{\varepsilon}' \leq p+1=p+\varepsilon$. Then $\underline{\varepsilon}' \leq r-r' < p+\varepsilon$. Let us show that $r-r' \geq 1= \varepsilon$. There are two possibilities for $\underline{\varepsilon}'$. Either $\underline{\varepsilon}'=1$, in which case $r-r' \geq 1= \varepsilon$ or  $\underline{\varepsilon}'=0$. If $\underline{\varepsilon}'=0$, then $r-r' \geq 0$. But actually  $r-r' \geq 1$, since if $r-r'= 0$, from the fact that $p \mid \underline{a}'=r'$ we would have that $p \mid r$, a contradiction.

\begin{figure}[h]
  \begin{subfigure}[b]{0.47\textwidth}
\[
\gyoung(!\fgg;;;!\fg;;;,!\fgg;!\fg;,!\fgg;,!\fg;,;,;)
\]
\caption{Partition $(6,2,1^4)$, $\textup{Rim}^*_3((6,2,1^4))$, and $\textup{Rim}^*_3((6,2,1^4)^{(1)*}).$}
\label{fig:ex5rim}
 
  \end{subfigure}
  \begin{subfigure}[b]{0.47\textwidth}
  \[
\gyoung(!\lt!\fgg;;;!\fg;;;,!\fgg;!\fg;!\lno,!\fgg;,!\fg;,;,;)
\]
\caption{Partition $\underline{(6,2,1^4)}$ in thicker lines.}
\label{fig:ex5rimsub}
	\end{subfigure}
	\caption{}
\end{figure}

\end{itemize}

\end{enumerate}

\end{proof}

\section{From self-Mullineux partitions to BG-partitions}\label{sect:surj}

From Theorem \ref{th:main}, we know that the notion of BG-symbol induces an algorithm for the correspondence between BG-partitions and self-Mullineux partition, since it defines an injective mapping between sets of the same cardinality. Then we know that to each Mullineux symbol of a self-Mullineux partition, corresponds a unique BG-partition. Moreover, from the definition of the BG-symbol, and Lemma \ref{lema1}, we know how to find the BG-partition associated to such a Mullineux symbol under this correspondence.
In this section we prove that this inverse algorithm is well defined, that is, we prove that applying it to a Mullineux symbol of a self-Mullineux partition results in a BG-partition. This confirms Theorem \ref{th:main} in a combinatorial way without using the fact that $\#\bg{p}{n}=\#M_p^n$.

\begin{proposition}\label{surjective} Let $p$ be an odd prime and $\lambda$ a self-Mullineux partition. The Mullineux symbol of $\lambda$, $G_p(\lambda)$ is the BG-symbol of some BG-partition. That is
\[
\mathcal{M}_p \subseteq \bgs{p}(\bg{p}{}).
\]
\end{proposition}

\begin{proof}

We give a combinatorial proof of this fact, although it follows also directly from Proposition \ref{injective} and Proposition \ref{prop:imagesymbol}. 

 We proceed by induction on $l$, the length of the Mullineux symbol. 

Let $l=0$ and $S= \begin{pmatrix}
a_l \\ r_l
\end{pmatrix} \in \mathcal{M}_p$, that is, $S=G_p(\lambda)$ for some $\lambda \in M_p$. Let 
$\varepsilon_l=0$ if $p\mid a_l$ and $\varepsilon_l=1$ otherwise. 
Since $S$ has exactly one column, then $\lambda=\lambda^{(l)}$ is a hook, that is $\lambda$ is of the form $\lambda=(x,1^y)$, with $x\leq p$. On the other hand, since $\lambda$ is fixed by the Mullineux map, we know that
\[
a_l=2r_l-\varepsilon_l.
\]
We claim that $\varepsilon_l=1$. If $\varepsilon_l=0$, that is, if $p \mid a_l$, then the $p$-segments that form $\textup{Rim}_p(\lambda)=[\lambda]$ are all of length exactly $p$. We know that $\lambda$ is a $p$-regular hook, this means that $\lambda$ is formed by exactly one $p$-segment. If there was more than one $p$-segment, then $a_l>p$ (so that $a_l=2kp$ for some $k\geq1$) it follows that $y\geq p$, and then $\lambda$ would not be $p$-regular. Thus $a_l=p=2r_l$. But this is not possible since $p$ is odd. Then $\varepsilon_l=1$ and $a_l=2r_l-1$.

The partition $\mu=(r_l,1^{r_l-1})$ is self-conjugate, and is a hook of length $2r_l-1=a_l$. Since $p \nmid a_l$, then $\mu \in \bg{p}{}$. Its BG-symbol is
\[
\bgs{p}{}(\mu)= \begin{pmatrix}
2r_l-1 \\ r_l
\end{pmatrix} = \begin{pmatrix}
a_l \\ r_l 
\end{pmatrix} = S.
\]
In fact $\mu=\lambda$.

Consider now $l>0$. Let 
\[ S=
\begin{pmatrix}
a_0 & a_1 & \cdots & a_l \\
r_0 & r_1 & \cdots & r_l
\end{pmatrix}
\]
be a symbol in $\mathcal{M}_p$ corresponding to a partition $\lambda$ in $M_p$. Consider the array 
\[ \tilde{S}=
\begin{pmatrix}
 a_1 & \cdots & a_l \\
 r_1 & \cdots & r_l
\end{pmatrix}.
\]
By definition, $\tilde{S}$ is the Mullineux symbol of the partition $\lambda^{(1)}$, obtained from $\lambda$ by removing the nodes on the $p$-rim. We know that $\lambda^{(1)}$ is fixed by the Mullineux map, given that this only depends on the columns of the symbol. Then $\lambda^{(1)} \in M_p$, and $\tilde{S}\in \mathcal{M}_p$. By induction, there exists a partition $\tilde{\mu} \in \bg{p}{}$ such that 
\[
\bgs{p}{}(\tilde{\mu})=\tilde{S}.
\]

We will apply Lemma \ref{lema1}. Let $\varepsilon_0=0$ if $p \mid a_0$, or $\varepsilon_0=1$ otherwise. Let $m=(r_0-\varepsilon_0) \ \textup{mod}\ p$.

Suppose that $\varepsilon_0=0$ and let us see that in this case $m=0$. Since  $\varepsilon_0=0$, then $p \mid a_0$. But $a_0=2r_0-\varepsilon_0=2r_0$, since $\lambda$ is a fixed point of the Mullineux map. Now, since $p$ is odd, then $p \mid r_0$ so that $m=(r_0-\varepsilon_0)\ \textup{mod}\ p =r_0\ \textup{mod}\ p =0$.

If $\varepsilon_0=1$, we have that $p \nmid a_0$. Therefore $p \nmid 2m+1$ since $2m+1 \equiv 2(r_0-\varepsilon_0)+1 \ \textup{(mod }p)$ and $2(r_0-\varepsilon_0)+1=2r_0-\varepsilon_0=a_0.$

Lemma \ref{lema1} implies that there exists a unique self-conjugate partition $\mu \in \bg{p}{}$ such that
\begin{enumerate}
\item[(i)] $a^*_\mu \equiv \varepsilon_0\ (\textup{mod}\ 2)$;
\item[(ii)] $r^*_\mu-\varepsilon^*_\mu \equiv m\ (\textup{mod}\ p)$ and 
\item[(iii)] $\mu^{(1)*}=\tilde{\mu}$.
\end{enumerate}

The condition $\mu^{(1)*}=\tilde{\mu}$ implies that
\[
\bgs{p}{}(\mu)=
	\begin{pmatrix}
	a^*_\mu &       \\
		    &\bgs{p}{}(\tilde{\mu})\\
	r^*_\mu &       
	\end{pmatrix} =
	\begin{pmatrix}
	a^*_\mu &       \\
		    &\tilde{S}\\
	r^*_\mu &       
	\end{pmatrix} = 
	\begin{pmatrix}
a^*_\mu & a_1 & \cdots & a_l \\
r^*_\mu & r_1 & \cdots & r_l
\end{pmatrix}.
\]

Let us see that in fact $\bgs{p}{}(\mu)=S$, that is, $a^*_\mu=a_0$ and $r^*_\mu=r_0$. Indeed, from $(i)$, $a^*_\mu$ is even if and only if $\varepsilon_0=0$, if and only if $p \mid a_\lambda$. But $a^*_\mu$ is even if and only if $p \mid a^*_\mu$, by Corollary \ref{corollary:epsilon}. This sequence of equivalences says that $\varepsilon_0=\varepsilon^*_\mu$.
Then, by (ii) we have that $r^*_\mu \equiv r_0\ (\textup{mod}\ p)$.

 Since $S \in \mathcal{M}_p$, then, from Proposition \ref{cinco}, we have, in particular
\begin{equation}
\label{eq:1}
\varepsilon_0 \leq r_0 - r_1 < p+\varepsilon_0.
\end{equation}
On the other hand, since $\mu \in \bg{p}{}$, then $\bgs{p}{}(\mu) \in \mathcal{M}_p$, by Proposition \ref{prop:imagesymbol}, so that, in particular we have
\begin{equation}
\label{eq:2}
\varepsilon^*_\mu \leq r^*_\mu - r_1 < p+\varepsilon^*_\mu.
\end{equation}

Substracting Equation \ref{eq:1} from Equation \ref{eq:2}, we get
\[
-p < r^*_\mu - r_0 < p,
\]
but since $p \mid r^*_\mu - r_0$ we can conclude that $r^*_\mu - r_0=0$, so that $r^*_\mu = r_0$. Therefore $a^*_\mu=2r^*_\mu - \varepsilon^*_\mu=2r_0-\varepsilon_0=a_0$, and
\[
\bgs{p}{}(\mu)=
\begin{pmatrix}
a_0 & a_1 & \cdots & a_l \\
r_0 & r_1 & \cdots & r_l
\end{pmatrix}=S.
\]
\end{proof}

\appendix
\section{}
\label{appendix}

\noindent Let $n \in \mathbb{N}$ and $p$ an odd prime. Consider the two following sets of partitions of $n$
\[
\bg{p}{n}=\{\lambda \mid \lambda \vdash n; \, \lambda=\lambda'\ \text{and}\ \lambda\ \text{has no diagonal $(p)$-hooks}\},\ \text{and}
\]
\[
M_p^n=\{\lambda \mid \lambda \in \Reg{p}{n}\ \text{and}\ \ m(\lambda)=\lambda\},
\]
where $m$ is the Mullineux map.

\begin{example}\label{exbijection}
Let $n=18$ and $p=3$. There are, in total, $385$ partitions of $18$. The partitions $\lambda$ of $18$ such that $\lambda=\lambda'$ are
\[
( 5, 4^3, 1 ),\,( 6, 5, 2^3, 1 ),\,( 7, 4, 2^2, 1^3 ),\,
( 8, 3, 2, 1^5),\,( 9, 2, 1^7)
\]
Among them, those with no diagonal $(3)$-hooks are 
\[
\bg{3}{18}=\{(6,5,2^3,1),\, (7,4,2^2,1^3),\, (9,2,1^7)\}
\]
There are $135$ $3$-regular partitions of $18$. Those which are fixpoints of the Mullineux map are
\[
M_3^{18}=\{ (7, 5, 2^2, 1^2 ),\, ( 9, 4^2, 1 ),\, ( 10, 4^2 ) \}
\]
\end{example}

In general the sets $M^n_p$ of partitions of $n$ fixed by the Mullineux map and $\bg{p}{n}$ of BG-partitions of $n$, have the same number of elements. Indeed, it is easy to see that the number of self-Mullineux partitions is equal to the number of \emph{$p$-regular conjugacy classes} of the symetric group $\Sn$, contained in $A_n$, which split into two different conjugacy classes of $A_n$. Here, a $p$-regular conjugacy class means that the order of its elements is not divisible by $p$. This follows by \cite[Proposition 2]{andrewsolsson}. Now, the number of $p$-regular splitting conjugacy classes of $\Sn$ is equal to the number of self-conjugate partitions with diagonal hook-lengths not divisible by $p$. This is straightforward by \cite[1.2.10]{jameskerber} and the standard bijection between partitions in odd distinct parts and self-conjugate partitions (explained in detail later in this appendix).

This appendix contains an alternative proof of the fact that $\# M^n_p = \#\bg{p}{n}$.

\begin{remark} The generating function for the cardinality of $\bg{p}{n}$ is 
\[
\prod_{\substack{i\,\geq 0 \\ p\, \nmid\, 2i+1}}(1+t^{2i+1}).
\]
\end{remark}

For proving the mentioned identity, we need make some remarks about the intersection of conjugacy classes in $\Sn$ with $A_n$ and about Brauer characters of $A_n$.
\medskip

\noindent \textbf{Splitting of conjugacy classes of $\Sn$.} \ \ Let $C$ be a conjugacy class of $\Sn$ of even permutations. That is $C \subseteq A_n$. Then one of the two following possibilities holds:
\begin{itemize}
\item $C$ is a conjugacy class in $A_n$, or
\item $C$ splits into two conjugacy classes in $A_n$.

\end{itemize}

In the second possibility, say $C=C_1 \sqcup C_2$, these two conjugacy classes have the same size. Moreover, conjugating such classes by any element of $\Sn \setminus A_n$ permutes them, that is, if $\sigma \in \Sn \setminus A_n$, then $\sigma C_1 \sigma^{-1}=C_2$ and $\sigma C_2 \sigma^{-1}=C_1$.
Furthermore, the conjugacy class $C$ splits if and only if the cycle type of elements in $C$ consists of different odd integers (\cite[1.2.10]{jameskerber}). We call a conjugacy class \emph{$p$-regular} when the order of its elements is not divisible by $p$.

The set of $p$-regular conjugacy classes of $\Sn$ contained in $A_n$ is then formed by two types of conjugacy classes:
\[
A \sqcup B,
\]
where $A$ is set of $p$-regular conjugacy classes of $\Sn$ of even permutations which are also conjugacy classes in $A_n$ and $B$ is the set conjugacy classes of $\Sn$ which split into two conjugacy classes in $A_n$. Hence, the set of $p$-regular conjugacy classes of $A_n$ is 
\[
A \sqcup \overline{B},
\]
where $\overline{B}$ consists of conjugacy classes coming for restriction of those conjugacy classes in $B$. These conjugacy classes in $\overline{B}$ come by pairs, in the sense that if $\sigma \in \Sn \setminus A_n$ and $\overline{C} \in \overline{B}$, then $\sigma \overline{C} \sigma^{-1} \in \overline{B}$ and $C=\overline{C}\, \cup\, \sigma \overline{C} \sigma^{-1}$ is a conjugacy class of $\Sn$ in $B$. Hence, a basis for the space of $\mathbb{C}$-valued functions defined on $p$-regular elements of $A_n$ and constant on conjugacy classes is
\[
\{ \mathbbm{1}_C \mid C \in A \}\  \sqcup\  \{ \mathbbm{1}_{\overline{C}} \mid \overline{C} \in \overline{B} \},
\]
or
\[
\{ \mathbbm{1}_C \mid C \in A \}\  \sqcup\  \{ \mathbbm{1}_{\overline{C}} , \mathbbm{1}_{\sigma \overline{C} \sigma^{-1}} \mid \overline{C} \cup \sigma \overline{C} \sigma^{-1} = C \in B \}.
\]
We claim that the set $\bg{p}{n}$ is in bijection with the set of $p$-regular conjugacy classes of $S_n$, contained in $A_n$ which split in two different conjugacy classes of $A_n$. Notice that the set of conjugacy classes of $\Sn$ with cycle type consisting of different odd integers is in bijection with self-conjugate partitions of size $n$. Indeed, a conjugacy class whose cycle type consists of different odd integers is associated to a unique partition $\lambda=(\lambda_1,\ldots,\lambda_r)$ of $n$ with $\lambda_1,\ldots,\lambda_r$ different odd integers (the lengths of the cycles in the cycle decomposition, in decreasing order). Consider the self-conjugate partition $\mu$ defined by the lengths of its diagonal hooks as follows: $h^\mu_{11}=\lambda_1,h^\mu_{22}=\lambda_2,\ldots,h^\mu_{rr}=\lambda_r$. The condition of $\lambda_i$'s being different, and then strictly decreasing, ensures that $\mu$ is a well defined partition. Conversely, any self-conjugate partition of $n$ corresponds to a unique finite sequence of different odd integers; the lengths of its hooks (\cite[2.5.11]{jameskerber}). 

Conjugacy classes of $\Sn$ with cycle type consisting of different odd integers are in particular contained in $A_n$. If we consider those conjugacy classes with the additional condition of being $p$-regular, which form in fact the set $B$, they are therefore in bijection with self-conjugate partitions such that $p$ does not divide the length of any diagonal hook, that is, the set $\bg{p}{n}$. Hence $B$ is in bijection with $\bg{p}{n}$.
\medskip

\noindent \textbf{Brauer characters of $A_n$.} \ \  Let $D$ be an irreducible $F A_n$-module. To $D$ we can associate a function $\chi_D$ which is called the \emph{(irreducible) Brauer character} of $A_n$ afforded by $D$. This function $\chi_D$ is a complexed-valued function defined on the set of $p$-regular elements of $A_n$ and it is constant on conjugacy classes. Furthermore, isomorphic $F A_n$-modules are associated to equal Brauer characters. See \cite[\textsection 15]{isaacs} for the precise definition of Brauer character and for further information.

In particular \cite[Theorem 15.10]{isaacs} says that the set of irreducible Brauer characters of $A_n$ form a basis of the space of $\mathbb{C}$-valued functions defined on $p$-regular elements of $A_n$ and constant on conjugacy classes. This implies that there are as many irreducible Brauer characters of $A_n$ as $p$-regular conjugacy classes of $A_n$, and by \cite[Corollary 15.11]{isaacs}, this is also the number of isomorphism classes of $F A_n$-modules. Therefore to each element $\mu$ of the set
\[
\{ \lambda \mid \lambda \vdash n,\ \text{$\lambda$ $p$-regular and}\   \lambda \neq m(\lambda) \}\  \sqcup \ \{\lambda^+,\,\lambda^- \mid \lambda \vdash n,\ \text{$\lambda$ $p$-regular and}\   \lambda = m(\lambda)\},
\]

which parametrizes irreducible $F A_n$-modules (see \ref{eqn:indxmodA} in the introduction), we can associate an irreducible Brauer character $\chi_{[\mu]}$. That way, a basis of the space of $\mathbb{C}$-valued functions defined on $p$-regular elements of $A_n$ and constant on conjugacy classes is 
\[
\{ \chi_{[\lambda]} \mid \lambda \in \Reg{p}{n} \ \ \text{and} \ \ \lambda \neq m(\lambda)\}\  \sqcup\  \{\chi_{[\lambda^+]},\chi_{[\lambda^-]} \mid \lambda \in \Reg{p}{n} \ \ \text{and}\ \ \lambda=m(\lambda)\},
\]
considering only one partition $\lambda$ for each couple $\{\lambda,m(\lambda)\}$ with $\lambda \neq m(\lambda)$.

\begin{proposition}\label{bijection}
The sets $M^n_p$ and $\bg{p}{n}$ have the same number of elements.
\end{proposition}

\begin{proof}
To prove this, we will give two bases of a same space of functions, and the equality of the cardinality of these bases will give the result.

Denote by $E$ the space of $\mathbb{C}$-valued functions defined on $p$-regular elements of $A_n$ and constant on conjugacy classes
\[
E=\left\{f:\{p\text{-regular elements of}\ A_n\}\longrightarrow\mathbb{C} \mid f\ \text{is a class function of } A_n \right\}.
\]

Define an action of $\Sn$ on $E$ by conjugation as follows: for $\sigma \in \Sn$ and $f \in E$, $f^\sigma$ is the class function
\[
f^\sigma(\tau):=f(\sigma \tau \sigma^{-1}).
\]
For $\sigma \in \Sn \setminus A_n$, let $E^\sigma$ be the set of class functions fixed by conjugation by $\sigma$:
\[
E^\sigma=\{f\in E \mid f^\sigma=f \}.
\]
This is a subspace of $E$. From the above discussion about splitting of conjugacy classes and how conjugation permutes some conjugacy classes, a basis for $E^\sigma$ is 
 
\[
\{ \mathbbm{1}_C \mid C \in A \}\  \sqcup\  \{ \mathbbm{1}_{\overline{C}} + \mathbbm{1}_{\sigma \overline{C} \sigma^{-1}} \mid \overline{C} \cup \sigma \overline{C} \sigma^{-1} = C \in B \}.
\]
We claim that a basis for $E^\sigma$ is 
\[
\{ \chi_{[\lambda]} \mid \lambda \in \Reg{p}{n} \ \ \text{and} \ \ \lambda \neq m(\lambda)\}\  \sqcup\  \{\chi_{[\lambda^+]}+\chi_{[\lambda^-]} \mid \lambda \in \Reg{p}{n} \ \ \text{and}\ \ \lambda=m(\lambda)\}.
\]
Indeed, this comes from the fact that, as with usual characters of representations in characteristic zero, conjugation of the character of a representation is the character of conjugation of the representation, here with Brauer characters. And also from the fact that conjugation by $\sigma$ permutes the modules associated to $\lambda^+ $ and $\lambda^-$ above.

Now, we have two bases for $E$ and two bases for $E^\sigma$. On one hand, from the characteristic function basis, the dimension of $E$ is
\[
\# A + \#\overline{B}= \# A + 2(\# B),
\]
and from the Brauer character basis, the dimension of $E$ is
\[
\#\{ \{\lambda,m(\lambda)\}\mid \lambda \in \Reg{p}{n} \ \ \text{and} \ \ \lambda \neq m(\lambda)\}\  + \  \#\{\lambda^+,\lambda^- \mid \lambda \in \Reg{p}{n} \ \ \text{and}\ \ \lambda=m(\lambda)\}.
\]
That is
\[\#D+2(\#M_p^n),\]
where $D=\{\lambda,m(\lambda)\}\mid \lambda \in \Reg{p}{n} \ \ \text{and} \ \ \lambda \neq m(\lambda)\}$. Hence,
\[
\# A + 2(\# B)=\#D+2(\# M_p^n).
\]

Counting the elements on the two bases for $E^\sigma$, we obtain that the dimension of $E^\sigma$ is
\[
\# A + \# B=\#D+\# M_p^n.
\]
These two identities imply that $\# B=\# M_p^n$. Since $B$ is in bijection with $\bg{p}{n}$ we obtain the result.

\end{proof}

\bibliographystyle{alpha}
\bibliography{bib}

\end{document}